\documentclass[10pt,a4paper,american]{amsart}
\usepackage{babel}
\usepackage{amsmath,amsthm,bm}
\usepackage{amsbsy}
\usepackage{amstext}
\usepackage{amsfonts}
\usepackage{floatflt}
\usepackage[pdftex]{graphicx}
\usepackage{mathcomp}
\usepackage{mathtools}
\usepackage{color}
\usepackage{dsfont}
\usepackage{latexsym}
\usepackage{amssymb}
\usepackage{stmaryrd}
\usepackage{bm}
\usepackage{enumerate}
\usepackage{esint}	
\usepackage[colorlinks,pdfpagelabels,pdfstartview = FitH,bookmarksopen = true,bookmarksnumbered = true,linkcolor = blue,plainpages = false,hypertexnames = false,citecolor = red,pagebackref=false]{hyperref}
\usepackage{cite}

\allowdisplaybreaks
\makeatletter
\newsavebox{\@brx}
\newcommand{\llangle}[1][]{\savebox{\@brx}{\(\m@th{#1\langle}\)}%
  \mathopen{\copy\@brx\kern-0.5\wd\@brx\usebox{\@brx}}}
\newcommand{\rrangle}[1][]{\savebox{\@brx}{\(\m@th{#1\rangle}\)}%
  \mathclose{\copy\@brx\kern-0.5\wd\@brx\usebox{\@brx}}}
\makeatother

\newtheorem{theorem}{Theorem}
\newtheorem{lemma}[theorem]{Lemma}
\newtheorem{proposition}[theorem]{Proposition}
\newtheorem{corollary}[theorem]{Corollary}

\theoremstyle{definition}
\newtheorem{definition}[theorem]{Definition}

\theoremstyle{remark}
\newtheorem{remark}[theorem]{Remark}

\numberwithin{theorem}{section}
\numberwithin{equation}{section}

\DeclareMathOperator{\dist}{dist}

\DeclareMathOperator*{\essliminf}{ess\,lim\,inf}

\DeclareMathOperator*{\essinf}{ess\,inf}

\newcommand{\N}{\ensuremath{\mathbb{N}}}
\newcommand{\R}{\ensuremath{\mathbb{R}}}

\newcommand{\spt}{\operatorname{spt}}

\newcommand{\mollifytime}[2]{[\![ #1 ]\!]_{#2}}

\newcommand{\dx}{\mathrm{d}x}
\newcommand{\dt}{\mathrm{d}t}

\makeatletter
\@namedef{subjclassname@2020}{%
  \textup{2020} Mathematics Subject Classification}
\makeatother

\def\Xint#1{\mathchoice
    {\XXint\displaystyle\textstyle{#1}}%
    {\XXint\textstyle\scriptstyle{#1}}%
    {\XXint\scriptstyle\scriptscriptstyle{#1}}%
    {\XXint\scriptscriptstyle\scriptscriptstyle{#1}}%
    \!\int}
\def\XXint#1#2#3{\setbox0=\hbox{$#1{#2#3}{\int}$}
    \vcenter{\hbox{$#2#3$}}\kern-0.5\wd0}
\def\bint{\Xint-}
\def\dashint{\Xint{\raise4pt\hbox to7pt{\hrulefill}}}

\def\XXiint#1#2#3{\setbox0=\hbox{$#1{#2#3}{\iint}$}
    \vcenter{\hbox{$#2#3$}}\kern-0.5\wd0}

\renewcommand{\epsilon}{\varepsilon}
\newcommand{\eps}{\varepsilon}
\renewcommand{\rho}{\varrho}

\renewcommand{\epsilon}{\varepsilon}
\renewcommand{\rho}{\varrho}
\renewcommand{\d}{\:\! \mathrm{d}}

\DeclareMathOperator{\loc}{loc}

\numberwithin{equation}{section}
\allowdisplaybreaks


\begin{document}
\renewcommand{\refname}{References} 
\renewcommand{\abstractname}{Abstract} 
\title[Obstacle problem for the PME]{On two notions of solutions to
  the obstacle problem for the singular porous medium equation}
\author[K. Moring]{Kristian Moring}
\address{Kristian Moring\\
	Fakult\"at f\"ur Mathematik, Universit\"at Duisburg-Essen\\
	Thea-Leymann-Str. 9, 45127 Essen, Germany}
\email{kristian.moring@uni-due.de}

\author[C. Scheven]{Christoph Scheven}
\address{Christoph Scheven\\
	Fakult\"at f\"ur Mathematik, Universit\"at Duisburg-Essen\\
	Thea-Leymann-Str. 9, 45127 Essen, Germany}
\email{christoph.scheven@uni-due.de}


\begin{abstract}
We show that two different notions of solutions to the obstacle
problem for the porous medium equation, a potential theoretic notion
and a notion based on a variational inequality, coincide for regular enough compactly supported obstacles. 
\end{abstract}
\makeatother

\maketitle

\section{Introduction}

In this paper we consider the porous medium equation (PME for short), which can be written as 
\begin{equation}\label{evo_eqn}
\partial_t u -\Delta( u^m)=0, 
\end{equation}  
for $0<m<\infty$ and nonnegative $u$. Here we focus on the singular range $0 < m < 1$. For the standard theory of this equation we refer to the monographs~\cite{Vazquez,Vazquez2,DK}.

In the obstacle problem, the objective roughly is to find a solution to the equation with an additional constraint. The constraint here is that the solution must stay above the given obstacle in the whole domain. Besides being interesting in its own right, the obstacle problem plays a role as a standard tool in nonlinear potential theory for example (see e.g.~\cite{HKM,KinnunenLindqvist2006,KinnunenLindqvist_crelle}). In this paper, we consider two different notions of solutions to the obstacle problem: weak solution and the minimal supersolution above the given obstacle. 

For the porous medium equation, the existence of weak (or variational) solutions has been studied in~\cite{AL,BLS,Schaetzler1,Schaetzler2}, and regularity questions are addressed in~\cite{MS,CS-obstacle,BLS-holder,MSb}. Typically, existence proofs for weak solutions require sufficient smoothness of the obstacle; in particular, regularity in time is a frequent presumption~\cite{BDM,AL,BLS,Schaetzler1,Schaetzler2}. On the contrary, existence and uniqueness of the minimal supersolution above the given obstacle are rather immediate consequences due to the concept of balayage from potential theory without any differentiability assumptions on the obstacle. 
It is also worth to mention that a more explicit construction of minimal supersolutions above the obstacle based on Schwartz alternating method has been proposed, see e.g.~\cite{korte_lscobstacle,kokusi}.

Main objective of the paper is to show a connection between the
notions of weak solutions to the obstacle problem and minimal
supersolutions above the given obstacle. In the case of the parabolic
$p$-Laplace equation, this direction has been studied
in~\cite{LP}. For the
porous medium equation in the slow diffusion case partial results were
given in~\cite{AvLu}, where the authors showed that the minimal
supersolution is a weak solution to the obstacle problem. In this
paper we are able to show that in the fast diffusion case these
notions actually coincide under appropriate assumptions. More
precisely, we suppose that the obstacle and its first weak derivatives
exist with appropriate integrability conditions such that the
existence of a weak solution is guaranteed
(\hspace{1sp}\cite{BLS,Schaetzler2}). Furthermore, we require that the
obstacle is H\"older continuous to ensure that weak solutions are
continuous (see~\cite{MS,CS-obstacle} and also~\cite{KMN,kokusi} for
the parabolic $p$-Laplace equation).

In Section~\ref{sec:varsol} we introduce weak solutions to the
obstacle problem and prove some properties that are beneficial later
on. More specifically, we show that outside the contact set a weak
solution to the obstacle problem is a weak solution to the obstacle
free problem (Lemma~\ref{l.variational-solution-noncontactset}). In
Section~\ref{sec:minimal} we consider minimal supersolutions above the
given obstacle, which we define via the concept of balayage. The
definition is based on the notion of supercaloric functions that has
been studied for the porous medium equation
e.g. in~\cite{KinnunenLindqvist_crelle,KLLP-supercal,MSch}. Many
useful results on supercaloric functions and their connection to weak
supersolutions in the fast diffusion case were established
in~\cite{MSch}. In particular, it was shown that in each connected
component of the domain the positivity set of a supercaloric function
can be expressed as a countable union of time intervals. Our result on
the connection of the two notions for obstacle problems relies heavily on this property.

The connection between the two notions is established in
Section~\ref{sec:connection}. The precise statement is given
in Theorem~\ref{t.weakvar-is-minimal}. The proof is closely connected to a suitable comparison principle (as in~\cite{AvLu}) together with results on positivity sets of supercaloric functions in the fast diffusion case established in~\cite{MSch}. The main idea is to show that a weak solution to the obstacle problem is below the minimal supersolution in the domain, after which the conclusion is rather immediate by the minimality property of the latter.

In the linear case, the comparison principle for general open sets is
a rather straightforward consequence once the comparison principle in
finite unions of space-time cylinders is at disposal. This is mainly
due to the fact that the class(es) of solutions is closed under
addition of constants. The porous medium equation does not enjoy this
property, which makes such a comparison principle more difficult to
verify. The idea proposed in~\cite{AvLu} was, instead of adding a
small constant to solution, to multiply the solution with a constant close to one. While the resulting function is not a weak solution in the original sense exactly, it still satisfies the equation with a certain source term, which disappears when the constant tends to one. In~\cite{AvLu}, this observation allowed to deduce a comparison principle between solutions and supersolutions in more general sets than space-time cylinders. More precisely, the result was shown in cylinders from which certain compact set can be removed. A crucial assumption that was made in the proof -- due to multiplication instead of addition of a constant -- was that the supersolution stays strictly positive in this compact set.

In order to adapt this comparison principle to the proof that weak solutions to the obstacle problem and minimal supercaloric functions coincide (Theorem~\ref{t.weakvar-is-minimal}), we rely on the fact that in the fast diffusion case positivity sets of supercaloric functions depend only on the time variable (\hspace{1sp}\cite{MSch}). Due to this property, we are able to circumvent the additional positivity assumption on the supersolution (minimal supersolution above the obstacle) and finally conclude the result.

\medskip
 
\noindent
{\bf Acknowledgments.} K.~Moring has been supported by the Magnus Ehrnrooth Foundation.

\section{Preliminaries}

Let $\Omega \subset \R^n$ be an open and bounded set. For
$T>0$ we denote by $\Omega_T := \Omega \times (0,T)$ a space-time
cylinder in $\R^{n+1}$. The parabolic boundary of $\Omega_T$ is defined as $\partial_p \Omega_T := \left( \Omega \times \{0\} \right) \cup \left( \partial \Omega \times [0,T) \right)$. Up next we define weak (super/sub)solutions.

\begin{definition} \label{d.weak_sol}
A measurable function $u: \Omega_T \to [0,\infty]$ satisfying
$$
u^m \in L_{\loc}^2(0,T;H^{1}_{\loc}(\Omega)) \cap L^\frac{1}{m}_{\loc}(\Omega_T)
$$
is called a weak solution to the PME~\eqref{evo_eqn} if and only if $u$ satisfies the integral equality
\begin{equation} \label{eq:wsol}
\iint_{\Omega_T} \left(-u \partial_t \varphi + \nabla u^m  \cdot \nabla \varphi \right)\d x\d t  =0
\end{equation}
for every $\varphi \in C^\infty_0(\Omega_T)$. Further, we say that $u$ is a weak supersolution if the integral above is nonnegative for all nonnegative test functions $\varphi \in C^{\infty}_0(\Omega_T)$. If the integral is nonpositive for such test functions, we call $u$ a weak subsolution.

Furthermore, we say that $u$ is a global weak solution if $u$ satisfies~\eqref{eq:wsol} and
$$
u^m \in L^2(0,T;H^1(\Omega)) \cap L^\frac{1}{m}(\Omega_T).
$$
\end{definition}

We recall an existence and comparison result for continuous, global weak solutions, see~\cite{Abdulla1,Abdulla2,Bjorns_boundary,MSch}.

\begin{theorem} \label{t.existence}
Let $0<m<1$ and $\Omega_T$ be a $C^{1,\alpha}$-cylinder
with $\alpha >0$. Suppose that the function $g \in C(\overline{\Omega_T})$ satisfies $g^m \in L^2(0,T;H^1(\Omega))$ and $\partial_t g^m \in L^\frac{m+1}{m}(\Omega_T)$. Then, there exists a unique global weak solution $u$ to~\eqref{evo_eqn} such that $u \in C(\overline{\Omega_T})$, $u$ is locally H\"older continuous and $u = g$ on $\partial_p \Omega_T$. Moreover, if $g'$ satisfies conditions above and $g\leq g'$ on $\partial_p \Omega$ and $h' \in C(\overline {\Omega_T})$ is a global weak solution with boundary values $g'$ on $\partial_p \Omega_T$, then $h \leq h'$ in $\Omega_T$.
\end{theorem}

The next result states that a weak supersolution is lower semicontinuous after possible redefinition in a set of measure zero, see~\cite{Naian}.

\begin{theorem} \label{t.super_lsc}
Suppose that $m > 0$. Let $u$ be a weak supersolution according to
Definition~\ref{d.weak_sol}. Then, there exists a lower semicontinuous
function $u_*$ such that $u_*(x,t) = u(x,t)$ for a.e. $(x,t) \in
\Omega_T$. Moreover, $u_*$ satisfies
\begin{equation}\label{def:u-star}
u_*(x,t) = \essliminf_{\substack{(y,s) \to (x,t) \\ s<t}} u(y,s) =
\lim_{\rho\to0}\left( \essinf_{B_\rho(x)\times (t-\rho^2,t)} u
\right).
\end{equation}
\end{theorem}

The following Caccioppoli inequality holds for bounded weak supersolutions,
see~\cite[Lemma 2.15]{KinnunenLindqvist_crelle}.

\begin{lemma} \label{l.bounded_caccioppoli}
Suppose that $m>0$ and $u \leq M$ is a weak supersolution in $\Omega_T$. Then, there exists a numerical constant $c > 0$ such that
$$
\int_{0}^{T}\int_{\Omega} \xi^2 \left| \nabla u^m  \right|^2 \, \d x \d t \leq  c M^{2m} T \int_\Omega \left| \nabla \xi \right|^2 \, \d x + c M^{m+1} \int_\Omega \xi^2 \, \d x
$$
for every $\xi = \xi(x) \in C_0^\infty (\Omega)$ with $\xi \geq 0$.
\end{lemma}

Then we recall the definition of super- and subcaloric functions.

\begin{definition} \label{d.supercal}
We call $u \colon \Omega_T \to [0,\infty]$ a supercaloric function, if the following properties hold:
\begin{itemize}
\item[(i)] $u$ is lower semicontinuous,
\item[(ii)] $u$ is finite in a dense subset,
\item[(iii)] $u$ satisfies the comparison principle in every subcylinder $Q_{t_1,t_2}=Q \times (t_1,t_2) \Subset \Omega_T$: if $h\in C(\overline{Q}_{t_1,t_2})$ is a weak solution in $Q_{t_1,t_2}$ and if $h \leq u$ on the parabolic boundary of $Q_{t_1,t_2}$, then $ h\leq u$ in $Q_{t_1,t_2}$.
\end{itemize}

A function $v: \Omega_T \to [0,\infty)$ is called subcaloric function
if the conditions (i), (ii) and (iii) above hold with (i) replaced by
upper semicontinuity, and the inequalities in (iii) by $\geq$.
\end{definition}

\begin{remark}
In~\cite{MSch} the authors considered so called quasi-super(sub)caloric functions, for which the comparison principle (iii) in Definition~\ref{d.supercal} is required only in $C^{2,\alpha}$-subcylinders of $\Omega_T$. It was shown that the classes of quasi-super(sub)caloric and super(sub)caloric functions coincide, which implies that it is sufficient to require that the comparison principle in item (iii) holds in all $C^{2,\alpha}$-subcylinders.
\end{remark}

Supercaloric functions are closed under increasing limits provided that the limit function is finite in a dense set, see~\cite[Proposition 4.6]{Bjorns_boundary}.

\begin{lemma} \label{l.superc_increasing_lim}
Let $m>0$ and $u_k$ be a nondecreasing sequence of supercaloric functions in $\Omega_T$. If $u := \lim_{k\to \infty} u_k$ is finite in a dense subset of $\Omega_T$, then $u$ is supercaloric in $\Omega_T$.
\end{lemma}

Furthermore, we can extend supercaloric functions by zero to the past, see~\cite{MSch}.

\begin{lemma} \label{l.zero-past-extension}
Let $0<m<1$ and $v : \Omega_T \to [0,\infty]$ be a supercaloric function in $\Omega_T$. Then
\[ u =
\begin{cases}
v\quad &\text {in } \Omega \times (0,T), \\
0\quad &\text {in } \Omega \times (-\infty,0],
\end{cases}
\]
is a supercaloric function in $\Omega \times (-\infty,T)$.
\end{lemma}

In the
fast diffusion case, on a fixed time slice a weak supersolution is
either positive or zero everywhere provided that the domain is connected, see~\cite{MSch}.

\begin{lemma}\label{lem:alternatives}
  Let $0<m<1$ and assume that $u$ is a supercaloric function to the PME in $\Omega_T$,
  where $\Omega\subset\R^n$ is open and
  connected. Then, for any time $t\in(0,T)$ either
  $u$ is positive on the whole time slice $\Omega\times\{t\}$ or $u$
  vanishes on the whole time slice.
  In particular, the positivity set of $u$ can be written as the union
  of cylinders $\Omega\times\Lambda_i$ for at most countably many open intervals
  $\Lambda_i\subset(0,T)$. 
\end{lemma}

We also have the following connections of supercaloric functions and weak supersolutions, see~\cite{MSch}.

\begin{proposition} \label{p.bdd-supercal-supersol}
Let $0<m<1$ and suppose that $u$ is a locally bounded supercaloric function to the porous medium equation in $\Omega_T$,
  where $\Omega\subset\R^n$ is an open set. Then, $u$ is a weak supersolution.
\end{proposition} 

In connection to the next result, see Theorem~\ref{t.super_lsc}. For the proof we refer to~\cite{MSch}.
\begin{lemma} \label{l.weasuper-is-supercal}
Let $0<m<1$ and suppose that $u$ is a weak supersolution in $\Omega_T$. Then, $u_*$ is a supercaloric function in $\Omega_T$.
\end{lemma}

For the proof of the next lemma we refer to~\cite{MSch}, see also~\cite[Theorem 3.6]{Bjorns_boundary} and~\cite[Theorem 3.3]{KLL}. 
\begin{lemma} \label{l.supersubcal-cylinder-comparison}
Let $m>0$ and $U_{t_1,t_2} \Subset \R^{n+1}$ be a cylinder. Suppose that $u$ is a supercaloric and $v$ is a subcaloric function in $U_{t_1,t_2}$. If
$$
\infty \neq \limsup_{U_{t_1,t_2} \ni (y,s) \to (x,t)} v(y,s) \leq \liminf_{U_{t_1,t_2} \ni (y,s) \to (x,t)} u(y,s)
$$
for every $(x,t) \in \partial_p U_{t_1,t_2}$, then $v \leq u$ in $U_{t_1,t_2}$.
\end{lemma}

For supercaloric functions the following result holds true by~\cite{MSch}.

\begin{theorem} \label{t.supercal-essliminf}
Let $0 < m< 1$ and $u : \Omega_T \to [0,\infty]$ be a supercaloric function in $\Omega_T$. Then,
$$
u(x,t) = \essliminf_{\substack{(y,s) \to (x,t) \\ s<t}} u(y,s)\quad \text{for every } (x,t) \in \Omega_T.
$$
\end{theorem}

\begin{definition} For $v \in L^1(\Omega_T)$, $v_o \in L^1(\Omega)$ and $h>0$, we define a mollification in time by
\begin{equation} \label{eq:time-mollif}
  \mollifytime{v}{h}(x,t)
  :=
  \mollifytime{v}{h,v_o}(x,t)
	:= e^{-\frac{t}{h}}v_o(x) + \tfrac1h \int_0^t e^{\frac{s-t}{h}}
        v(x,s)\, \mathrm{d}s
 \end{equation}
 for $(x,t)\in \Omega\times[0,T]$.
\end{definition}

We state some useful properties of the mollification in the following lemma, see~\cite[Lemma 2.2]{KinnunenLindqvist2006}.

\begin{lemma} \label{lem:mollifier}
  Let $\mollifytime{v}{h}$ be defined as in~\eqref{eq:time-mollif}. Then the following properties hold:
  \begin{enumerate}[(i)]
  \itemsep2mm
  \item Let $p\ge 1$ and $X\in \{L^p(\Omega),W^{1,p}(\Omega),W^{1,p}_0(\Omega)\}$. If $v \in L^p(0,T;X)$ and $v_o \in X$,
    then $\mollifytime{v}{h}\in C([0,T];X)$ and $\mollifytime{v}{h}(\cdot,0)=v_o$.
    Furthermore $\mollifytime{v}{h} \xrightarrow{h\to 0} v$ in $L^p(0,T;X)$ .
\item If $v \in C(\overline{\Omega_T})$, $v_o = v(\cdot,0)$ and $\Omega \subset \R^n$ is a bounded set, then 
$$
\mollifytime{v}{h} \xrightarrow{h\to 0} v\quad \text{ uniformly in }  \Omega_T.
$$
\item If $v\in L^p(\Omega_T)$ for some $p\ge1$,
  the weak time derivative $\partial_t \mollifytime{v}{h}$ exists in
  $L^p(\Omega_T)$  and is given by the formula
\begin{equation*}
  \partial_t \mollifytime{v}{h} = \frac{1}{h} ( v - \mollifytime{v}{h} ).
\end{equation*}
\end{enumerate}
\end{lemma}

\section{Weak solution to the obstacle problem} \label{sec:varsol}

We consider an obstacle function $\psi:\Omega\times[0,T]\to\R_{\ge0}$ and
abbreviate $\psi_o:=\psi(\cdot,0)$. 
Now we define the notion of weak solution to the obstacle problem. 
Let
$$
K_\psi(\Omega_T) := \{ v \colon
v^m \in L^2(0,T;H^{1}(\Omega)),\, u\in C([0,T];L^{m+1}(\Omega)), v \geq \psi \text{ a.e. in } \Omega_T \}
$$
and
$$
K'_\psi(\Omega_T) := \{ v \in K_\psi(\Omega_T) : \partial_t (v^m) \in L^\frac{m+1}{m}(\Omega_T) \}.
$$
For cutoff functions
$\eta\in C^1_0(\Omega, \R_{\geq 0})$ and $\alpha \in
W^{1,\infty}_0((0,T);\R_{\geq 0})$ we define
\begin{align*}
\llangle \partial_t u, \alpha \eta (v^m - u^m) \rrangle &:= \iint_{\Omega_T} \eta \left[ \alpha' \left( \tfrac{1}{m+1} u^{m+1} - uv^m \right) -\alpha u \partial_t v^m \right] \, \d x \d t.
\end{align*}

\begin{definition} \label{d.local-weak-obstacle}

A function $u \in K_\psi(\Omega_T)$ is a local weak solution to the obstacle problem for the PME if
\begin{equation} \label{e.local_var_eq}
\llangle \partial_t u, \alpha \eta (v^m - u^m) \rrangle + \iint_{\Omega_T} \alpha \nabla u^m \cdot \nabla \left(\eta( v^m - u^m) \right) \d x \d t \geq 0
\end{equation}
holds true for all comparison maps $v \in K'_\psi(\Omega_T)$, every
$\eta\in C^1_0(\Omega,\R_{\ge0})$ and $\alpha \in W^{1,\infty}_0((0,T);\R_{\geq 0})$.
\end{definition}

For a weak solution, we additionally require that it agrees with the
obstacle on the parabolic boundary.

\begin{definition} \label{d.variational-obstacle}

A function $u \in K_{\psi}(\Omega_T)$ is a weak solution to
the obstacle problem for the PME with $u=\psi$ on $\partial_p\Omega_T$
if it is a local weak solution in the sense of
Definition~\ref{d.local-weak-obstacle}, it attains the 
initial values $u(\cdot,0)=\psi_o$ a.e. in $\Omega$ and
$u^m - \psi^m \in L^2(0,T;H^1_0(\Omega))$.
\end{definition}

We will rely on the following results on existence and regularity of
weak solutions to the obstacle problem. For the existence, we refer to \cite{BLS,Schaetzler2}, see also Appendix~\ref{appendix-a}. Regularity follows from~\cite{CS-obstacle,MS}. 
\begin{theorem}\label{t.cont-exist}
  Let $\Omega \subset \R^n$ be a bounded Lipschitz domain and $\psi^m\in
  L^2(0,T;H^1(\Omega))\cap C([0,T];L^{\frac{m+1}{m}}(\Omega))$ with $\partial_t(\psi^m)\in
  L^{\frac{m+1}{m}}(\Omega_T)$ and $\psi_o^m \in
  L^{\frac{m+1}{m}}(\Omega) \cap H^1(\Omega)$. Then, there exists a weak
  solution $u$ to the obstacle problem with $u=\psi$ on
  $\partial_p\Omega_T$ in the sense of Definition~\ref{d.variational-obstacle}.
  
  If the obstacle additionally satisfies $\psi\in C^{0;\beta,\frac{\beta}{2}}(\Omega_T)$ for some $\beta \in (0,1)$,
  then every locally bounded local weak solution to the  obstacle problem is locally H\"older continuous in $\Omega_T$. 
\end{theorem}

We show that every local weak solution to the obstacle problem is a weak supersolution to the obstacle free problem.

\begin{lemma} \label{l.variationalsol-is-supersol}
Suppose that $\Omega \Subset \R^n$ and let $u \in K_\psi(\Omega_T)$ be a local weak solution according to Definition~\ref{d.local-weak-obstacle} to the obstacle problem with obstacle $\psi \in C(\overline{\Omega_T})$. Then $u$ is a weak supersolution.
\end{lemma}

\begin{remark}
One could alternatively assume that $\psi^m \in L^2(0,T;H^1(\Omega))$ and $\partial_t \psi^m \in L^\frac{m+1}{m}(\Omega_T)$ instead of continuity of $\psi$.
\end{remark}

\begin{proof}
Let $\varphi \in C^\infty_0(\Omega_T, \R_{\geq 0})$. In the definition
of local weak solution, we consider
cutoff functions $\eta\in C^1_0(\Omega,\R_{\ge0})$ and $\alpha \in
W^{1,\infty}([0,T];\R_{\geq 0})$ with $\eta \alpha
\equiv 1$ in $\spt (\varphi)$ such that $\alpha$ is compactly supported in
$(0,T)$. Let $R \geq 1$ such that $B_R(0) \supset \Omega$, and extend
$\psi(\cdot,0)$ as continuous function to $B_{2R} (0) \setminus
\overline{\Omega}$ and as zero to $\R^n \setminus B_{2R}(0)$. Denote
this extension by $\psi_o$. We will use time mollifications
$\mollifytime{u^m}{h,(\psi_o^m)_\delta}$ and
$\mollifytime{\psi^m}{h,(\psi_o^m)_\delta}$. Observe that for each
sequence $\eps_i \xrightarrow{i\to \infty} 0$ there exists a sequence
$\delta_i \xrightarrow{i \to \infty} 0$ such that
$$
\|\mollifytime{\psi^m}{h,(\psi_o^m)_{\delta}} -
\mollifytime{\psi^m}{h,\psi_o^m}\|_{L^\infty(\Omega_T)}
=
\|(\psi_o^m)_{\delta} -\psi_o^m\|_{L^\infty(\Omega)}
< \eps_i
$$ for
every $\delta \in (0,\delta_i]$,
uniformly in $h>0$, 
and by Lemma~\ref{lem:mollifier}\,(ii) there is a sequence $\hat h_i
\xrightarrow{i \to \infty}0$ such that
$$
  \|\mollifytime{\psi^m}{h,\psi_o^m} - \psi^m\|_{L^\infty(\Omega_T)} <
  \eps_i
$$
for every $h \in (0,\hat h_i]$. Furthermore, for $\eps_i$ and
$\delta_i$ constructed above, Lemma~\ref{lem:mollifier}\,(i) yields a sequence $\tilde h_i
\xrightarrow{i \to \infty} 0$ such that
$$
  \max\{\| \mollifytime{u^m}{h,(\psi_o^m)_{\delta_i}} - u^m
  \|_{L^\frac{m+1}{m}(\Omega_T)},
  \|\mollifytime{u^m}{h,(\psi_o^m)_{\delta_i}} -
  u^m\|_{L^2(0,T;H^1(\Omega))} \} < \eps_i
$$
for each $h \in (0,\tilde h_i]$. By choosing $h_i := \min\{\hat h_i, \tilde h_i\}$ we have that
\begin{align*}
  \begin{array}{rl}
  \mollifytime{\psi^m}{i} :=
  \mollifytime{\psi^m}{h_i,(\psi_o^m)_{\delta_i}} \xrightarrow{i \to
  \infty } \psi^m&\mbox{ uniformly in $\Omega_T$ and}\\[0.8ex]
  \mollifytime{u^m}{i} := \mollifytime{u^m}{h_i,(\psi_o^m)_{\delta_i}} \xrightarrow{i \to \infty } u\phantom{^m}&\mbox{ in $L^2(0,T;H^1(\Omega))$ and in $L^\frac{m+1}{m}(\Omega_T)$}.
  \end{array}
\end{align*}
Define a comparison map 
$$
v^m_i = \mollifytime{u^m}{i} + \varphi + \|\psi^m - \mollifytime{\psi^m}{i}\|_\infty.
$$
It follows that $v^m_i \in K'_\psi(\Omega_T)$ and in particular,
$$
v^m_i \geq \mollifytime{u^m}{i} + \psi^m - \mollifytime{\psi^m}{i}  \geq \psi^m \quad \text{ a.e. in } \Omega_T,
$$
since $\varphi \geq 0$ and $u \geq \psi$ a.e. in $\Omega_T$. 

For the parabolic part in~\eqref{e.local_var_eq} we have
\begin{align*}
  &\llangle \partial_t u , \alpha \eta (v^m_i -u^m) \rrangle \\
  &\qquad= \iint_{\Omega_T} \eta \alpha' \left(\tfrac{1}{m+1}u^{m+1} - u (\mollifytime{u^m}{i} + \varphi + \|\psi^m - \mollifytime{\psi^m}{i}\|_\infty )\right) \, \d x \d t \\
&\qquad\phantom{+} - \iint_{\Omega_T} \alpha \eta u \partial_t \left( \mollifytime{u^m}{i} + \varphi + \|\psi^m - \mollifytime{\psi^m}{i}\|_\infty \right) \, \d x \d t. 
\end{align*}
Since Lemma~\ref{lem:mollifier} (iii) implies
  \begin{align}\label{time-deriv-moll}
    u\partial_t\mollifytime{u^m}{i}
    &=
    \big(u-\mollifytime{u^m}{i}^{\frac1m}\big)\partial_t\mollifytime{u^m}{i}
    +
    \mollifytime{u^m}{i}^{\frac1m}\partial_t\mollifytime{u^m}{i}\\\nonumber
    &\ge
    \mollifytime{u^m}{i}^{\frac1m}\partial_t\mollifytime{u^m}{i}
    =
    \tfrac{m}{m+1}\partial_t\mollifytime{u^m}{i}^{\frac{m+1}{m}},
  \end{align}
we may estimate 
\begin{align*}
- \iint_{\Omega_T} \alpha \eta u \partial_t \mollifytime{u^m}{i} \, \d x \d t &\leq
    \tfrac{m}{m+1}\iint_{\Omega_T} \alpha' \eta \mollifytime{u^m}{i}^\frac{m+1}{m}  \, \d x \d t.
\end{align*}
By combining the
estimates above, and using Lemma~\ref{lem:mollifier} (i), (ii) together with the fact $\alpha\eta \equiv 1$ in $\spt(\varphi)$, we have
\begin{align*}
\limsup_{i\to\infty}\llangle \partial_t u , \alpha \eta (v_i^m -u^m) \rrangle &\leq - \iint_{\Omega_T} \eta u( \alpha' \varphi + \alpha \partial_t \varphi) \, \d x \d t \\
&= -\iint_{\Omega_T} u \partial_t \varphi \, \d x \d t.
\end{align*}
For the divergence part we obtain
\begin{align*}
\iint_{\Omega_T} &\alpha \nabla u^m \cdot \nabla(\eta (v^m_i - u^m) ) \, \d x \d t\\ 
&= \iint_{\Omega_T} \alpha (\mollifytime{u^m}{i} - u^m ) \nabla u^m \cdot \nabla\eta \, \d x \d t +\iint_{\Omega_T} \alpha \varphi \nabla u^m \cdot \nabla\eta \, \d x \d t \\
&\phantom{+}+\iint_{\Omega_T} \alpha \|\psi^m - \mollifytime{\psi^m}{i}\|_\infty \nabla u^m \cdot \nabla \eta\, \d x\d t \\
&\phantom{+}+ \iint_{\Omega_T} \alpha \eta  \nabla u^m \cdot \nabla(\mollifytime{u^m}{i} - u^m ) \, \d x \d t+ \iint_{\Omega_T} \alpha \eta  \nabla u^m \cdot \nabla \varphi \, \d x \d t. 
\end{align*}
Again, by using the fact $\alpha \eta \equiv 1$ in $\spt(\varphi)$ and Lemma~\ref{lem:mollifier} (i), (ii) we arrive at
$$
\lim_{i\to\infty} \iint_{\Omega_T} \alpha \nabla u^m \cdot \nabla(\eta (v^m_i - u^m) ) \, \d x \d t = \iint_{\Omega_T} \nabla u^m \cdot \nabla \varphi \, \d x \d t.
$$
By combining the results, we obtain
\begin{equation*}
  \iint_{\Omega_T} (-u\partial_t\varphi+\nabla u^m \cdot \nabla \varphi) \, \d x \d t\ge0,
\end{equation*}
and the claim follows.
\end{proof}

The next lemma shows that local weak solutions to the obstacle problem
are weak solutions to the PME outside of the contact set. 

\begin{lemma} \label{l.variational-solution-noncontactset}
  Let $\psi\in C(\overline{\Omega_T})$, and
  $u\in K_\psi(\Omega_T)$ be a continuous, local weak solution to the obstacle problem
  for the porous medium equation in the sense of
  Definition~\ref{d.local-weak-obstacle}.
  Then, $u$ is a weak solution to the
  porous medium equation
  in the set $\{(x,t)\in\Omega_T\colon u(x,t)>\psi(x,t)\}$. 
\end{lemma}

\begin{proof}
  It suffices to prove that $u$ is a weak solution in any parabolic
  cylinder $Q=B\times(t_o,t_1)\Subset\{u>\psi\}$.
  Without loss generality, we may assume that $u^m(\cdot,t_o)\in
  H^1(B)$.
  We consider an arbitrary test function $\varphi\in
  C^\infty_0(Q)$. In~\eqref{e.local_var_eq}, we choose cutoff functions 
  $\alpha\in W^{1,\infty}_0((t_o,t_1),[0,1])$ and $\eta\in
  C^1_0(B,[0,1])$ with $\alpha\eta\equiv1$ in $\spt(\varphi)$ and
  $\spt(\alpha\eta)\Subset Q$.
  We use the time mollification $\mollifytime{u^m}{h}$ 
  introduced in~\eqref{eq:time-mollif}
  on the domain $Q$ instead of $\Omega_T$, with initial values
  $u^m(\cdot,t_o)\in H^1(B)\cap L^{\frac{m+1}{m}}(B)$.
  Then we define a comparison map as
  \begin{equation*}
    v^m_h:= \mollifytime{u^m}{h}+\varepsilon\varphi 
  \end{equation*}
  with appropriately chosen $\eps>0$.  Since $u\in C(\overline Q)$, we have $\mollifytime{u^m}{h}\xrightarrow{h\downarrow0} u^m$ uniformly in
  $Q$ by Lemma~\ref{lem:mollifier} (ii), so that we may assume
  \begin{equation*}
    \mollifytime{u^m}{h}>u^m-d \mbox{\quad in $Q$, \qquad with }
    d:=\tfrac12\inf_{Q}(u^m-\psi^m)>0,
  \end{equation*}
  by choosing $h>0$ sufficiently small. If $\varphi\ge0$ in $Q$, we
  consider an arbitrary $\eps>0$. Otherwise, we restrict ourselves to
  parameters 
  \begin{equation*}
    0<\eps<\frac{d}{-\inf\varphi}.
  \end{equation*}
  In any case, we obtain $v^m_h>u^m-2d\ge\psi^m$ in $Q$. We note that by our choices of the cutoff
    functions, it is sufficient to ensure the obstacle constraint in
    $Q$. In fact, instead of $v_h^m$ we can use the comparison map
    \begin{equation*}
      v^m:=\zeta v_h^m+(1-\zeta)\sup_{\Omega_T}\psi^m,
    \end{equation*}
    with a cutoff function $\zeta\in C^\infty_0(Q,[0,1])$ satisfying
    $\zeta\equiv1$ on $\spt(\alpha\eta)$. Since $v_h\ge\psi$ a.e. in
    $Q$, we have $v\ge\psi$ a.e. in $\Omega_T$. Therefore, we can use
    $v$ in the variational inequality~\eqref{e.local_var_eq}. By our
    choice of $\zeta$, this gives the same result as plugging in
    $v_h$, i.e.
  \begin{align}\label{var-ineq}
    0&\le
       \langle\!\langle\partial_tu,\alpha\eta(v^m_h-u^m)\rangle\!\rangle
    +
    \iint_{\Omega_T}\alpha \nabla u^m\cdot
       \nabla\big(\eta(v^m_h-u^m)\big) \, \dx\dt\\\nonumber
    &=
     \iint_{\Omega_T} \eta \big( \alpha'\big(\tfrac1{m+1}u^{m+1}-u\mollifytime{u^m}{h}\big)
      -\alpha u\partial_t\mollifytime{u^m}{h} \big) \, \dx\dt\\\nonumber
     &\qquad
       -\eps\iint_{\Omega_T}\eta\big( \alpha'u\varphi+\alpha
       u\partial_t\varphi\big) \, \dx\dt\\\nonumber       
     &\qquad+
       \iint_{\Omega_T}\alpha \nabla u^m\cdot
       \nabla\big(\eta(\mollifytime{u^m}{h}-u^m+\varepsilon\varphi
       )\big) \, \dx\dt\\\nonumber
    &=:\mathrm{I}+\mathrm{II}+\mathrm{III}.
  \end{align}
  Using estimate~\eqref{time-deriv-moll}, 
  we can bound the first term by
  \begin{align*}
    \mathrm{I}
    &\le
    \iint_{\Omega_T} \eta \big( \alpha'\big(\tfrac1{m+1}u^{m+1}-u\mollifytime{u^m}{h}\big)
      -\alpha \tfrac{m}{m+1}\partial_t\mollifytime{u^m}{h}^{\frac{m+1}{m}}
      \big)\dx\dt\\
    &=
      \iint_{\Omega_T} \eta \alpha'\big(\tfrac1{m+1}u^{m+1}-u\mollifytime{u^m}{h}
      +\tfrac{m}{m+1}\mollifytime{u^m}{h}^{\frac{m+1}{m}}\big)\dx\dt\\
    &\xrightarrow{h\downarrow0} 0.
  \end{align*}
  The last convergence follows from Lemma~\ref{lem:mollifier}\,(i). 
  Moreover, since $\alpha\eta\equiv1$ on
  $\spt(\varphi)$, we have
  \begin{align*}
    \mathrm{II}
    =
    -\eps\iint_{\Omega_T}u\partial_t\varphi \, \dx\dt,
  \end{align*}
  and Lemma~\ref{lem:mollifier}\,(i) implies
  \begin{align*}
    \mathrm{III}
    \xrightarrow{h\downarrow0}
    \eps\iint_{\Omega_T}\alpha\nabla u^m\cdot\nabla(\eta\varphi)\,\dx\dt
    =
    \eps\iint_{\Omega_T}\nabla u^m\cdot\nabla\varphi \, \dx\dt.
  \end{align*}
  Therefore, by letting $h\downarrow0$
  in~\eqref{var-ineq}, we infer
  \begin{equation*}
    0\le\eps\iint_{\Omega_T}\big(-u\partial_t\varphi+\nabla
    u^m\cdot\nabla\varphi\big) \, \dx\dt.
  \end{equation*}
  Since the same inequality holds for $-\varphi$ in place of
  $\varphi$, we actually have equality. After dividing by $\eps>0$, we
  thus obtain
    \begin{equation*}
    \iint_{\Omega_T}\big(-u\partial_t\varphi+\nabla u^m\cdot\nabla\varphi\big) \, \dx\dt=0
  \end{equation*}
  for every $\varphi\in C^\infty_0(Q)$, for an arbitrary cylinder
  $Q\Subset\{u>\psi\}$. A partition of unity argument now implies that
  $u$ is a weak solution to the porous medium equation in the set 
  $\{u>\psi\}$. 
\end{proof}

\section{Minimal supersolution above the obstacle} \label{sec:minimal}

We define the minimal supersolution above the given obstacle as follows.

\begin{definition} \label{d.minimal-supercal}
Let $\psi \in C(\Omega_T) \cap L^\infty(\Omega_T)$. We say that $u: \Omega_T \to [0,\infty)$ is a minimal supersolution above the obstacle $\psi$ if the following properties hold true:
\begin{itemize}
\item[(i)] $u(x,t) \geq \psi(x,t)$ for every $(x,t) \in \Omega_T$;
\item[(ii)] $u$ is a supercaloric function in $\Omega_T$;
\item[(iii)] $u$ is the smallest supercaloric function in $\Omega_T$ which lies above $\psi$, i.e., if $v$ is a supercaloric function with $v(x,t)\geq \psi(x,t)$ for all $(x,t) \in \Omega_T$, then $v(x,t)\geq u(x,t)$ for all $(x,t) \in \Omega_T$;
\item[(iv)] $u$ is a weak solution to~\eqref{evo_eqn} in the set $\{(x,t) \in \Omega_T : u(x,t) > \psi (x,t) \}$.
\end{itemize}  
\end{definition}

Next we define the balayage of the given obstacle $\psi$, which will be our candidate for the solution in Definition~\ref{d.minimal-supercal}.

\begin{definition}
Let $\psi: \Omega_T \to [0,\infty]$ be a function, and denote
$$
\mathcal U_\psi := \{ v \text{ is a supercaloric function in } \Omega_T : v(x,t) \geq \psi(x,t)\, \text{ for every } (x,t) \in \Omega_T \}.
$$
We define the r\'eduite of $\psi$ as 
$$
R_\psi(x,t) = \inf \{ v(x,t): v \in \mathcal U_\psi \},
$$
and the balayage of $\psi$ as its lower semicontinuous regularization
$$
\hat R_\psi(x,t) := \lim_{\rho \to 0} \left( \inf_{B_\rho(x) \times (t-\rho^2,t+\rho^2)} R_\psi \right). 
$$

\end{definition}
From the definition it clearly follows that $\hat R_\psi$ is a lower semicontinuous function. Observe that, in general, $\hat R_\psi$ is not necessarily above $\psi$ at every point. However, for example if $\psi$ is lower semicontinuous, then $\hat R_\psi \geq \psi$ everywhere in $\Omega_T$.

\begin{remark}
By the argument in~\cite[Lemma 2.7]{LP} together with
Theorem~\ref{t.supercal-essliminf} we have that $\hat R_\psi =
(R_\psi)_*$ everywhere in $\Omega_T$, where $(R_\psi)_*$ denotes the
$\essliminf$-regularization of $R_\psi$ as introduced
in~\eqref{def:u-star}. Thus it does not matter which regularization we
use for the r\'eduite $R_\psi$. Moreover, if $\psi$ is a bounded
function, then it follows that $\hat R_\psi = R_\psi$ a.e. in
$\Omega_T$, see~\cite[Theorem 2.8]{LP} in connection with
Proposition~\ref{p.bdd-supercal-supersol},
Theorem~\ref{t.super_lsc} and Lemma~\ref{l.bounded_caccioppoli}.
\end{remark}

\begin{lemma} \label{l.balayage-supercal}
Let $\psi : \Omega_T \to [0,\infty)$ be a bounded, lower semicontinuous function in $\Omega_T$. Then $\hat R_\psi$ exists, it is bounded and satisfies properties (i), (ii) and (iii) in Definition~\ref{d.minimal-supercal}.
\end{lemma}

\begin{proof}
Since $\sup_{\Omega_T} \psi < \infty$ is a supercaloric
function as a constant, it follows that $\sup_{\Omega_T} \psi \in
\mathcal U_\psi$. Hence, the set $\mathcal U_\psi$ is nonempty and $\hat R_\psi \leq R_\psi \leq \sup_{\Omega_T} \psi < \infty$, which implies existence and boundedness of $\hat R_\psi$.

Clearly $R_\psi(x,t) \geq \psi(x,t)$ for every $(x,t) \in \Omega_T$ and since $\psi$ is lower semicontinuous, also $\hat R_\psi(x,t) \geq \psi(x,t)$ for every $(x,t) \in \Omega_T$. This shows property (i).

To show that $\hat R_\psi$ is supercaloric, we are left to ensure the
comparison principle, item (iii) in Definition~\ref{d.supercal}. Let
$Q \Subset \Omega_T$ be a cylinder and $h \in C(\overline{Q})$ be a
weak solution in $Q$ such that $h \leq \hat R_\psi $ on $\partial_p
Q$. This implies $h \leq v$ on $\partial_p Q$ for any $v\in \mathcal
U_\psi$, and since $v$ is supercaloric, we deduce
$h \leq v$ in $Q$. Since $v\in \mathcal U_\psi$ is arbitrary, this implies that $h \leq R_\psi$ in $Q$. Since $h$ is continuous, it follows that $h \leq \hat R_\psi$ in $Q$ completing the proof of property (ii).

To show item (iii) in Definition~\ref{d.minimal-supercal}, we observe
that for every $v\in\mathcal U_\psi$, we have $\psi\le\hat R_\psi\le
R_\psi\le v$ everywhere in $\Omega_T$. Hence, $\hat R_\psi$ is the smallest
supercaloric function above $\psi$. 
\end{proof}

\begin{proposition}\label{p.balayage-supercal}
Let $\psi \in C(\Omega_T) \cap L^\infty(\Omega_T)$. Then, $\hat R_\psi$ is the minimal supersolution above the obstacle $\psi$ according to Definition~\ref{d.minimal-supercal}.
\end{proposition}
\begin{proof}
Properties (i), (ii) and (iii) in Definition~\ref{d.minimal-supercal} follow directly from Lemma~\ref{l.balayage-supercal}. Thus we are left to prove the property (iv).

We denote $u:= \hat R_\psi$. Clearly, the set $\{u>\psi\}$ is open by lower semicontinuity of $u$ and continuity of $\psi$. Thus for fixed $(x_o,t_o) \in \{u>\psi\}$ we can find $\lambda > 0$, $\delta > 0$ and a $C^{2,\alpha}$-cylinder $U_{t_1,t_2} \Subset \Omega_T$ with $(x_o,t_o)\in U_{t_1,t_2}$ such that
$$
u>\lambda > \psi\quad \text{ in } \overline{U_{t_1,t_2+\delta}}.
$$
Let $\phi_j$ be a sequence of nonnegative Lipschitz functions in
$\overline{U_{t_1,t_2+\delta}}$ with the property $\lambda^m \leq \phi_j \leq
\phi_{j+1}\leq u^m$, and $\lim_{j\to \infty} \phi_j(x,t) = u^m(x,t)$
for all $(x,t) \in \overline{U_{t_1,t_2+\delta}}$. Let $h_j$ be a global weak
solution in $U_{t_1,t_2+\delta}$ taking continuously boundary values
$(\phi_j)^\frac{1}{m}$ on $\partial_p U_{t_1,t_2+\delta}$.
This solution exists according to Theorem~\ref{t.existence}. We use it
to define a Poisson modification
\[ u_P^j := \begin{cases}
h_j\quad &\text{ in } U \times(t_1,t_2] \\
u\quad &\text{ in } \Omega_T \setminus ( U \times (t_1,t_2]).
\end{cases}
\]
Observe that since the boundary values $(\phi_j)^\frac{1}{m}$ are
increasing, Theorem~\ref{t.existence} implies that the sequence
$h_j$ is increasing and thus $u_P^j$ is increasing. Hence, we may define 
$$
u_P := \lim_{j\to \infty} u_P^j.
$$
Since $u$ is bounded, it follows that $h_j$ is uniformly bounded in
$\overline{U_{t_1,t_2+ \delta}}$. Since the sequence $\{h_j\}$ is also
increasing, we may conclude that the limit $h = \lim_{j\to \infty}
h_j$ is a continuous weak solution in $U_{t_1,t_2 + \delta}$ by~\cite[Theorem 18.1, Chapter 6]{DGV} together with Lemma~\ref{l.bounded_caccioppoli}. 
We claim that $u_P$ is a supercaloric function. Clearly $u_P$ is lower
semicontinuous in $ \Omega_T \setminus \partial U_{t_1,t_2}$ by
continuity of $h$ in $U_{t_1,t_2}$ and lower semicontinuity of
$u$. Moreover, $u_P$ is lower semicontinuous on $\partial_p U_{t_1,t_2}$ since $h$ is lower semicontinuous in $\overline U_{t_1,t_2}$ as an increasing limit of continuous functions and on $\overline {\partial_p U_{t_1,t_2}}$ the limit is $u$. On the slice $U \times \{t_2\}$ we have $h \leq u$, which together with lower semicontinuity of $u$ concludes that $u_P$ is lower semicontinuous in $\Omega_T$. Next we show the comparison principle.

Let $Q \Subset \Omega_T$ such that $Q \cap U_{t_1,t_2} \neq \varnothing$
and $Q \setminus U_{t_1,t_2} \neq \varnothing$, and let $\bar h \in
C(\overline{Q})$ be a weak solution in $Q$ with $\bar h \leq u_P$ on
$\partial_p Q$. Since $u_P\le u$ and $u$ is supercaloric, this
immediately implies that $\bar h \leq u$ in $Q$. Now we are left to
show that $\bar h \leq u_P = h$ in $Q \cap U_{t_1,t_2}$ and on the slice $Q \cap (U \times \{t_2\})$. In the former set,
$h$ is a supercaloric function as a continuous weak solution, see
Lemma~\ref{l.weasuper-is-supercal}, and similarly,  $\bar h$ is
subcaloric. Now we are able to use
Lemma~\ref{l.supersubcal-cylinder-comparison} to conclude that also
$\bar h \leq h$ in $Q\cap U_{t_1,t_2}$. Continuity of $h$ and $\bar h$ imply that this also holds on the slice $Q \cap (U \times \{t_2\})$. Altogether we have shown $\bar
h\le  u_P$ in $Q$. Thus $u_P$ is a supercaloric function in
$\Omega_T$.

Since $u_P^j \leq u$ in $\Omega_T$ for every $j \in \N$, it follows
that $u_P \leq u$. Moreover, since $u_P$ satisfies $u_P > \lambda$ on
$\partial_p U_{t_1,t_2}$, we may use the comparison principle in the definition of supercaloric functions to conclude that $u_P > \lambda$ in $U_{t_1,t_2}$.
Therefore, we have that 
$$
u \geq u_P > \lambda > \psi \quad \text{ in } U_{t_1,t_2}.
$$
Since $u = \hat R_\psi \leq v$ in $\Omega_T$ for any $v \in \mathcal U_\psi$, it follows that in particular $u \leq u_P$ in $\Omega_T$. This implies that $u = u_P$, and further that $u$ is a weak solution in $U_{t_1,t_2}$ since $u_P$ is. Since $(x_o,t_o)\in \{u>\psi\}$ was arbitrary, this holds in a neighbourhood of any point in $\{u > \psi\}$, proving property (iv).
\end{proof}

Next we show that for a bounded, compactly supported obstacle the
minimal supercaloric function has zero boundary values. In particular,
it belongs to the appropriate Sobolev space. Here we will assume that
$\Omega$ is a $C^{1,\alpha}$-domain. This is used to ensure the
existence of the Poisson modification by applying
Theorem~\ref{t.existence}.

The following proof is partly based on the ideas in~\cite{LP}.

\begin{lemma} \label{l.minimal-supercal-zerobv}
Let $\Omega$ be a $C^{1,\alpha}$-domain. Let $\psi$ be a bounded, lower semicontinuous obstacle
with compact support in $\Omega_T$.
Then, $\hat R_\psi^m \in L^2(0,T; H^1_0(\Omega))$. Furthermore, there exists $\delta >0$ such that $\hat R_\psi(\cdot,t) \equiv 0$ for every $t \in (0,\delta]$. 
\end{lemma}

\begin{proof}
Let $\delta >0$ such that $\psi(\cdot,t) \equiv 0$ for every $t \in (0,\delta]$. Such $\delta$ exists since $\psi$ has compact support in $\Omega_T$. Let $u$ denote the zero extension of $\hat R_\psi$ to the times $(0,\delta]$, i.e., $u = \hat R_\psi$ whenever $t \in (\delta, T)$ and $u \equiv 0$ when $t \in (0,\delta]$. This implies $u \leq \hat R_\psi$. Furthermore, $u$ is supercaloric in $\Omega_T$ by Lemma~\ref{l.zero-past-extension} and clearly $u \geq \psi$ everywhere in $\Omega_T$, which implies $\hat R_\psi \leq u$. Thus, $u = \hat R_\psi$.

We know that $\hat R_\psi$ is a weak supersolution by Proposition~\ref{p.bdd-supercal-supersol} since $\hat R_\psi \leq \sup_{\Omega_T} \psi < \infty$ everywhere in $\Omega_T$. This implies $\hat R_\psi^m \in L^2_{\loc}(0,T;
H^1_{\loc}(\Omega) )$ and further $\hat R_\psi^m \in L^2(0,T;
H^1_{\loc}(\Omega) )$ by the Caccioppoli inequality for bounded weak
supersolutions, Lemma~\ref{l.bounded_caccioppoli}.
Up next, we show $\hat R_\psi^m \in L^2(0,T; H^1_0(\Omega) )$.

We prove this by using a Poisson modification, and denote $u:=\hat
R_\psi$. Let $D\Subset\Omega$ be a $C^{1,\alpha}$-domain such that 
$\spt(\psi) \subset D \times (0,T)$ and  $(\Omega\setminus
\overline D)\times(0,T)$ is a $C^{1,\alpha}$-cylinder. Furthermore, let $U \subset \R^n$ be a $C^{1,\alpha}$-domain such that $D \Subset U \Subset \Omega$.
Let $\phi_j $ be a sequence of nonnegative Lipschitz functions such
that $\phi_j \leq \phi_{j+1} \leq u^m \chi_{U \times (0,T)}$ and $\phi_j \xrightarrow{j \to
  \infty} u^m$ pointwise everywhere in $U \times (0,T)$. Define Poisson modifications
\[
u_P^j :=
\begin{cases} 
u & \text{ in } \overline{D} \times (0,T),\\
h_j & \text{ in } (\Omega \setminus \overline{D}) \times (0,T),
\end{cases}
\]
in which $h_j$ is a continuous weak solution with boundary values
\[ 
h_j := 
\begin{cases} 
0 (=\phi_j^\frac{1}{m}) & \text{ on } \partial \Omega \times (0,T)\\
0 (=\phi_j^\frac{1}{m}) & \text{ on } (\overline{\Omega} \setminus D ) \times \{0\} \\
\phi_j^\frac{1}{m} & \text{ on } \partial D \times (0,T).
\end{cases}
\]
We claim that $u_P := \lim_{j\to \infty} u_P^j$ is a supercaloric
function with $u_P \leq u$ and $u_P^m \in L^2(0,T;
H^1_0(\Omega))$. Observe that by Theorem~\ref{t.existence} the
sequence $h_j$ is increasing so that the limit $\lim_{j\to \infty}
h_j$ exists.  

Clearly $u_P \leq u$ in $\overline{D} \times (0,T)$ by definition. By Lemma~\ref{l.supersubcal-cylinder-comparison} we have that $u_P^j \leq u$ in $(\Omega \setminus \overline{D}) \times (0,T)$ for every $j\in \N$. This also holds in the limit $j \to \infty$, which implies $u_P \leq u$ everywhere.

To show that $u_P$ is supercaloric, fix a $C^{2,\alpha}$-cylinder $Q_{t_1,t_2} \Subset \Omega_T$ (and suppose it intersects both regions where $U_P$ is defined differently) and let $\bar h \in C(\overline{Q_{t_1,t_2}})$ be a weak solution in $Q_{t_1,t_2}$ such that $\bar h \leq u_P$ on $\partial_p Q_{t_1,t_2}$. This immediately implies that $\bar h \leq u$ in $Q_{t_1,t_2}$ since $u$ is supercaloric, which takes care of the part $Q_{t_1,t_2} \cap [\overline{D} \times (0,T)]$. Then consider $Q_{t_1,t_2} \setminus [\overline{D} \times (0,T)]$. Since we know that $h= \lim_j h_j$ is a continuous weak solution in $(\Omega \setminus \overline{D} ) \times (0,T)$, we have that $\bar h \leq h$ on $\partial_p Q_{t_1,t_2} \cap [(\Omega \setminus \overline{D} ) \times (0,T)]$. Since $h$ is an increasing limit of continuous functions in $(\overline{\Omega} \setminus D ) \times [0,T)$, we have that the limit is lower semicontinuous (and continuous in the interior). This implies that for any point $(x,t) \in \partial D \times [t_1,t_2)$ we have
$$
\liminf_{ (\Omega \setminus \overline{D}) \times (t_1,t_2)  \ni (y,s) \to (x,t)} h(y,s) \geq h(x,t) = u(x,t) \geq \bar h(x,t),
$$
which implies that we can use Lemma~\ref{l.supersubcal-cylinder-comparison} to conclude that also $\bar h \leq h$ in $Q_{t_1,t_2} \setminus [\overline{D} \times (0,T)]$.
In total, we have now that $u_P \leq u$ and $u_P$ is a supercaloric
function in $\Omega_T$.
Also, $u_P \geq \psi$ clearly, which implies $u = u_P$ by minimality of $u$.
Now $u_P$ is a weak solution in $(\Omega\setminus\overline D)\times(0,T)$. 

We are left to show that $u^m \in L^2(0,T; H^1_0(\Omega))$. We know
that $h_j^m(\cdot,t)$ has zero boundary values in Sobolev sense on
$\partial \Omega$ for a.e. $t \in (0,T)$ and $h_j(\cdot,0)
\equiv 0$ for every $j \in \N$. 
Since $h_j \leq \| h\|_\infty < \infty$ for all $j \in \N$, we can derive a
Caccioppoli inequality for weak solutions $h_j$ in the form
\begin{align}\label{eq:caccio-weak}
\iint_{(\Omega \setminus D)\times (0,T)} &\eta^2 |\nabla h_j^m|^2 \, \d x \d t \leq c
\|h\|_{\infty}^{2m} T \int_{\Omega \setminus D} |\nabla \eta|^2 \, \d x 
\end{align}
with a numerical constant $c > 0$ and for any $\eta \in
C^\infty(\overline{\Omega} \setminus \overline{D})$ which vanishes in a neighbourhood of $\partial D$. This can be obtained by testing the mollified weak formulation of $h_j$ with test function $- \alpha(t) \eta^2(x) h_j^m(x,t)$, where $\alpha$ is a cutoff function in time with $\alpha(T) = 0$ (cf. Lemma~\ref{l.bounded_caccioppoli}).
We define $U_\lambda = \{x \in \overline{\Omega} \setminus
\overline{D} : \text{dist} (x, \partial D) > \lambda \}$ for
$\lambda\in(0,\frac14\dist(D,\partial\Omega))$. We apply~\eqref{eq:caccio-weak} with a function
$\eta_\lambda$ such that $\eta_\lambda \equiv 1$ in $U_\lambda$ and
$\eta_\lambda \equiv 0$ in $(\Omega \setminus D) \setminus
U_{\lambda/2}$. With this choice, estimate~\eqref{eq:caccio-weak}
implies that $|\nabla h_j^m|$ is uniformly bounded in $L^2(U_\lambda
\times (0,T))$, which in turn implies that there exists a subsequence
converging weakly, i.e., $\nabla h_j^m \rightharpoonup \nabla h^m$
weakly in $L^2(U_\lambda \times (0,T))$. Furthermore $\eta_{4\lambda}
h_j^m \rightharpoonup \eta_{4\lambda} h^m$ weakly in $L^2(0,T;
H^1(U_\lambda))$. Since $\eta_{4\lambda} h_j^m \in
L^2(0,T;H^1_0(U_\lambda))$, by Hahn-Banach theorem it follows that
also the weak limit satisfies $\eta_{4\lambda} h^m \in L^2(0,T;H^1_0(U_\lambda))$. Since $\eta_{4\lambda} \equiv 1$ in $U_{4\lambda}$, this implies that $u^m = u_P^m \in L^2(0,T;H^1_0(\Omega))$. 
\end{proof}

\section{Connection between the two notions} \label{sec:connection}

Throughout this section we will assume that $\Omega$ is a
$C^{2,\alpha}$-domain. This assumption is made to be able to apply the
duality type proof in the comparison principles used in Theorem~\ref{t.weakvar-is-minimal}.

We will use a comparison principle in a union of space time
cylinders. Let $\left \{U^i_{t_1^i,t_2^i}\right\}_i$ be a finite
collection of open space time cylinders. The lateral boundary of the union of such cylinders is denoted by
$$
\mathcal S \big(\cup_i U^i_{t_1^i,t_2^i} \big) = \Big(\! \cup_i
  \partial U^i\times [t_1^i,t_2^i]  \Big) \setminus \left( \cup_i U^i_{t_1^i,t_2^i} \right).
$$
The tops of the union we denote by
$$
\mathcal T \big( \cup_i U^i_{t_1^i,t_2^i} \big) = \left( \cup_i \overline{U^i} \times \{t_2^i\}\right) \setminus \left( \cup_i U^i_{t_1^i,t_2^i} \right),
$$
and the bottoms by
$$
\mathcal B \big( \cup_i U^i_{t_1^i,t_2^i} \big) = \left( \cup_i \overline{U^i} \times \{t_1^i\}\right) \setminus \left( \cup_i U^i_{t_1^i,t_2^i} \right).
$$

We start with a technical covering lemma.
\begin{lemma} \label{l.covering}
  Let $\Omega\subset\R^n$ be a $C^{2,\alpha}$-domain and
  $S<T$. Consider two sets
  $K,D\subset\Omega\times(S,T)$ such that $K$ has positive
  distance to $D$ and to the lateral boundary
  $\partial\Omega\times[S,T]$. Then there exist $C^{2,\alpha}$-domains
  $U_k\subset\Omega$, $k=1,\ldots,m$, and times
  $S=t_0<t_1<\ldots<t_m=T$ such that
  \begin{equation}\label{cover-D}
    D\subset \bigcup_{k=1}^m U_k\times(t_{k-1},t_k] =:D^F
    \qquad\mbox{and}\qquad  \dist(D^F,K)>0. 
  \end{equation}
  If $D_1\subset D_2\subset\ldots\subset\Omega\times(S,T)$
  is an increasing sequence of sets
  with $\dist(K,D_\ell)>0$ for each $\ell\in\N$, then the
  corresponding coverings can be chosen as an increasing sequence as
  well, i.e. $D_1^F\subset D_2^F\subset\ldots$. 
\end{lemma}

\begin{proof}
  We abbreviate $d:=\tfrac13\dist(K,D\cup(\partial\Omega\times[S,T]))>0$.
  For $m:=\lceil \frac{T-S}{d}\rceil\in\N$ we let $t_k:=S+kd$ for
  $k=0,\ldots,m-1$ and $t_m:= T$. Then we define 
  \begin{equation*}
    \widetilde U_k
    :=
    \{x\in\Omega\colon (x,s)\in D\mbox{ for some }s\in(t_{k-1},t_k]\},
  \end{equation*}
  for $k=1,\ldots,m$. We claim that
  \begin{equation}\label{pos-dist-U-tilde}
    \dist\big(K,\widetilde U_k\times(t_{k-1},t_k]\big)\ge 2d>0.
  \end{equation}
  For the proof of this estimate, we consider a point $(x,t)\in
  \widetilde U_k\times(t_{k-1},t_k]$. By
  definition of $\widetilde U_k$, there exists a time $s\in (t_{k-1},t_k]$
  with $(x,s)\in D$. Therefore, for any $z\in K$, we can estimate
  \begin{equation*}
    |(x,t)-z|\ge |(x,s)-z|-|t-s|\ge \dist(K,D)-(t_k-t_{k-1})\ge 3d-d=2d,
  \end{equation*}
  which proves \eqref{pos-dist-U-tilde}.
  Next, we claim that there exist $C^{2,\alpha}$-domains $U_k$ with 
  $\widetilde U_k\subset U_k\subset\Omega$ and
  \begin{equation}\label{pos-dist-U}
    \dist\big(K,U_k\times(t_{k-1},t_k]\big)\ge d>0.
  \end{equation}
  For the construction of these sets, we consider a mollification $g_\ast\in C^\infty(\Omega)$ of
  $g(x):=\dist(x,\widetilde U_k\cup\partial\Omega)$ with
  $\|g-g_\ast\|_{C^0}<\frac d4$. By Sard's theorem~\cite[Theorem 10.1, Chapter 2]{CH-book}, the critical
  values of $g_\ast$ form a zero set. Therefore, for a.e. $r\in(\frac
  d4,\frac{3d}{4})$, the level sets $\{x\in\Omega\colon g_\ast(x)=r\}$
  are smooth submanifolds. We fix a value $r\in(\frac
  d4,\frac{3d}{4})$ with this property and define
  \begin{equation*}
    U_k:=\big\{x\in\Omega\colon g_\ast(x)<r\big\}.
  \end{equation*}
  The boundary $\partial U_k$ is the disjoint union
  of the smooth submanifold $\{x\in\Omega\colon g_\ast(x)=r\}$
  and $\partial\Omega$. Since $\Omega$ is a $C^{2,\alpha}$-domain, the
  sets $U_k$ are $C^{2,\alpha}$-domains as well.
  Moreover, by definition the set $U_k$ is
  contained in the $d$-neighbourhood of $\widetilde U_k\cup\partial\Omega$. Therefore,
  estimate~\eqref{pos-dist-U-tilde} and the definition of $d$ 
  imply~\eqref{pos-dist-U}. 
 This proves claim~\eqref{cover-D} for a single set $D$.

  In the case of an increasing sequence $D_1\subset D_2\subset\ldots$,
  we start by constructing the covering $D_1^F\supset D_1$ in the same
  way as above. Assuming that the coverings $D_1^F\subset
  D_2^F\subset\ldots\subset D_{\ell}^F$ have already been
  chosen for some $\ell\in\N$, we apply the preceding construction to the set $D_{\ell+1}\cup
  D_{\ell}^F$ instead of $D$ to construct a covering
  $D_{\ell+1}^F\supset D_{\ell+1}\cup D_{\ell}^F$ with the asserted
  properties. This procedure leads to an increasing sequence of
  coverings $D_1^F\subset D_2^F\subset\ldots$ such
  that~\eqref{cover-D} is satisfied for $D_\ell^F$ in place of $D^F$.  
\end{proof}

In the following we suppose that $\psi$ is a nonnegative function satisfying 
\begin{equation} \label{a.obstacle_weak}
\psi^m \in L^2(0,T;H^1_0(\Omega)),\quad \partial_t(\psi^m) \in L^\frac{m+1}{m}(\Omega_T),
\end{equation}
and furthermore, $\psi$ is H\"older continuous with compact support in the following sense.
\begin{equation} \label{e.holder-obstacle}
\psi \in C^{0;\beta,\frac{\beta}{2}}_0(\Omega_T)\quad \text{ for some } \beta \in (0,1).
\end{equation}

Assumption \eqref{a.obstacle_weak} guarantees existence
and~\eqref{e.holder-obstacle} continuity of a weak solution to the
obstacle problem with $u=0$ on $\partial_p\Omega_T$ by
Theorem~\ref{t.cont-exist}. Observe that $u$ is bounded in $\Omega_T$
under these assumptions by~\cite[Lemma 2.9]{MSb}.
Note that the condition~\eqref{e.holder-obstacle} was assumed
in~\cite{CS-obstacle,MS} in order to prove that any locally bounded
weak solution is locally (H\"older) continuous as stated in
Theorem~\ref{t.cont-exist}. However, any weaker assumption on the
obstacle $\psi$ instead of~\eqref{e.holder-obstacle} that guarantees
continuity of a weak solution would be sufficient, see
Remark~\ref{rem:thm}.
Existence of the minimal
supersolution above the obstacle is guaranteed by
Proposition~\ref{p.balayage-supercal}, and its uniqueness is clear by
definition.
Then we show that if the obstacle satisfies conditions above, weak solutions to the obstacle problem and minimal supersolutions above the obstacle coincide.

\begin{theorem} \label{t.weakvar-is-minimal}
Let $0<m<1$, $\Omega \Subset \R^n$ be a $C^{2,\alpha}$-domain and $\psi$ an obstacle satisfying~\eqref{a.obstacle_weak}
and~\eqref{e.holder-obstacle}. Then, a weak solution $u$
to the obstacle problem
with $u=0$ on $\partial_p\Omega_T$ in the
sense of Definition~\ref{d.variational-obstacle}
and the minimal supersolution $v$ above
the obstacle in the sense of Definition~\ref{d.minimal-supercal} coincide, i.e., $u = v$ in $\Omega_T$.
\end{theorem}

\begin{remark} \label{rem:thm}
In principle, instead of~\eqref{a.obstacle_weak}
and~\eqref{e.holder-obstacle} one can merely assume that $\psi \in
C_0(\Omega_T)$. Then, the result holds whenever a weak solution $u\in C(\Omega_T)$ to the obstacle problem with zero boundary values and such obstacle exists.
\end{remark}

\begin{proof}

We start by proving the result by using a similar approach as in the
proof of~\cite[Theorem 3.1]{AvLu}.
Because of our assumptions~\eqref{a.obstacle_weak}
and~\eqref{e.holder-obstacle}, Theorem~\ref{t.cont-exist} implies that
the weak solution $u$ to the obstacle problem is continuous. Moreover, it is
a weak supersolution by Lemma~\ref{l.variationalsol-is-supersol} and a supercaloric function by
  Lemma~\ref{l.weasuper-is-supercal}. Therefore, and  since we  are
  considering the fast  diffusion range $0<m<1$,
  Lemma~\ref{lem:alternatives} guarantees that the positivity set of
  $u$ is the union of at most countably many sets
  of the form $\Omega^j\times\Lambda_i^j$ for open time intervals $\Lambda_i^j\subset(0,T)$ for each connected component $\Omega^j$ of $\Omega$. We consider the connected components $\Omega^j$ separately and for simplicity we omit $j$ from the superscript.
Let $(t_1,t_2) \subset (0,T)$ be one of these positivity intervals,
such that $u>0$ in $\Omega \times (t_1,t_2)$ and $u(\cdot,t_1) \equiv u(\cdot,t_2) \equiv0$, unless $t_2 = T$. Observe that this implies that $\psi(\cdot,t_1) \equiv \psi(\cdot,t_2) \equiv 0$. 

By Proposition~\ref{p.balayage-supercal}, the minimal supersolution
above the obstacle is given by $v = \hat R_\psi$, which is bounded
since $v \leq \sup_{\Omega_T} \psi < \infty$, and $v^m \in
L^2(0,T;H^1_0(\Omega))$ by Lemma~\ref{l.minimal-supercal-zerobv}. For
$s \in(t_1,t_2)$, let $\tau_1^i$ be a decreasing sequence such that $s>\tau_1^i \to t_1$ as $i \to \infty$. Let us consider $Q_i = \Omega \times (\tau_1^i, s)$  and 
$$
D_{\eps,s}^i = \left\{(x,t)\in Q_i : \frac{u(x,t)}{1+\eps} \geq v(x,t) \right\},
$$
for any fixed $\eps>0$. 
Observe that $Q_i \subset Q_{i+1}$ and $D_{\eps,s}^i \subset
D_{\eps,s}^{i+1}$. Define $K_i := \{(x,t) \in Q_i : u(x,t) = \psi(x,t)
\}$. Observe that in the set $K_i$ we have $0 < u = \psi \leq v$ and
$u$ is a weak solution in $Q_i \setminus K_i$ by
Lemma~\ref{l.variational-solution-noncontactset}. Let us denote
$u_\eps := u/(1+\eps)$. Now the function $u_\eps - v$ is upper
semicontinuous, which implies that $D_{\eps,s}^i$ is closed in $Q_i$,
and there is a positive distance between $D_{\eps,s}^i$ and $K_i$, as
well as between $\mathcal S (\Omega_T)$ and $K_i$. Therefore, for each
fixed $\eps>0$ and $i \in \N$, by Lemma~\ref{l.covering} we find a
collection of a finite number of $C^{2,\alpha}$-cylinders,
denoted by $D_{\eps,s}^{F,i}$, that covers $D_{\eps,s}^i$ and does not
intersect $K_i$. We fix an arbitrary $\eps_o>0$ and a 
  parameter $\eps\in(0,\eps_o]$. For $i\in\N$, we can choose the coverings
  $D_{\eps,s}^{F,i}$ such that $D^{F,1}_{\eps_o,s} \subset
  D^{F,i}_{\eps,s}\subset D^{F,i+1}_{\eps,s}$ for any $i\in\N$.
 
Since $u$ is a weak solution to the PME in $Q_i \setminus K_i$, we have
\begin{align*}
\iint_{Q_i \setminus K_i} -u_{\eps}\partial_t \varphi + \nabla u_{\eps}^m\cdot \nabla \varphi \, \d x\d t &= \iint_{Q_i \setminus K_i} - \tfrac{1}{1+\eps} u\partial_t \varphi + \tfrac{1}{(1+\eps)^m} \nabla u^m\cdot \nabla \varphi \, \d x\d t \\
&= \tfrac{1-(1+\eps)^{m-1}}{(1+ \eps)^m} \iint_{Q_i \setminus K_i}  \nabla u^m\cdot \nabla \varphi \, \d x\d t \\
&= \iint_{Q_i \setminus K_i}  \nabla f\cdot \nabla \varphi \, \d x\d t
\end{align*}
for every $\varphi \in C^\infty_0(Q_i \setminus K_i)$, in which we denoted $f = \tfrac{1-(1+\eps)^{m-1}}{(1+ \eps)^m} u^m$.
We consider an arbitrary nonnegative test function
\begin{equation}\label{choice-test-function}
  \phi\in C^\infty_0\Big(\overline{D_{\eps_o,s}^{F,1}}\cap (\Omega\times\{s\}),\R_{\ge0}\Big),
\end{equation}
which we extend to a function $\phi\in C^\infty(\Omega\times(0,s],\R_{\ge0})$ with
$\phi=0$ on $\Omega\times(0,s]\setminus \overline{D_{\eps_o,s}^{F,1}}$. 
Since $D_{\eps_o,s}^{F,1}\subset D_{\eps,s}^{F,i}$ for any $i\in\N$,
this choice implies in particular $\phi=0$ on $\partial D_{\eps,s}^{F,i}\setminus(\Omega\times\{s\})$. 
From now on, we omit the index $i$ in the superscript.
We consider a function $h = C^{2,1}(\overline{D_{\eps,s}^F},\R_{\geq 0})$ with
the boundary  values
\begin{align*}
  \left\{
  \begin{array}{cl}
  h=\phi&\qquad\mbox{on $\mathcal T(D_{\eps,s}^F)$},\\[1ex]
  h=0&\qquad\mbox{on $\mathcal S(D_{\eps,s}^F)$},
  \end{array}
  \right.
\end{align*}
which will later be chosen as the solution of a dual problem. Note
that as a consequence of the boundary condition and the fact $h\ge0$,
the outer normal derivative satisfies $\partial_n h \leq 0$ on
$\mathcal S(D_{\eps,s}^F)$.

Denote $b = u_\eps - v$. We have that $b \leq 0$ on $\partial D_{\eps,s}^F \cap Q_i$. On $\partial \Omega \times (\tau_1^i, s)$ we have $0 = u = u_\eps  = v$, i.e., $u_\eps^m - v^m = 0$ in the sense of traces. On $\Omega \times \{\tau_1^i\}$ and $\Omega \times \{ s \}$ we do not have information on the sign of $b$, only that $u,u_\eps > 0$ and $v \geq 0$ on these sets. Since $u_\eps$ is a weak solution with source term $f$ and $v$ is a weak supersolution in $D_{\eps,s}^F$, we can subtract the (very) weak formulations and obtain
\begin{align*}
\iint_{D_{\eps,s}^F} \nabla f \cdot \nabla h \, \d x \d t &\geq \int_{\mathcal T(D_{\eps,s}^F)} b \phi \, \d x - \int_{\mathcal B(D_{\eps,s}^F)} bh \, \d x - \iint_{D_{\eps,s}^F} b h_t \, \d x \d t \\
&\phantom{+} - \iint_{D_{\eps,s}^F} (u_\eps^m - v^m) \Delta h \, \d x \d t + \iint_{\mathcal S(D_{\eps,s}^F)} (u_\eps^m - v^m) \partial_n h \, \d \sigma \d t \\
&\geq \int_{\mathcal T(D_{\eps,s}^F)} b \phi \, \d x - \int_{\mathcal B(D_{\eps,s}^F)} bh \, \d x - \iint_{D_{\eps,s}^F} b h_t \, \d x \d t \\
&\phantom{+} - \iint_{D_{\eps,s}^F} (u_\eps^m - v^m) \Delta h \, \d x \d t,
\end{align*}
which can be written as
\begin{align} \label{e.comparison-est}
\int_{\mathcal T(D_{\eps,s}^F)} b \phi \, \d x &\leq \iint_{D_{\eps,s}^F} b(h_t + a \Delta h) \, \d x \d t  + \iint_{D_{\eps,s}^F} \nabla f \cdot \nabla h \, \d x \d t \\
&\phantom{+} + \int_{\mathcal B(D_{\eps,s}^F)} bh \, \d x \nonumber
\end{align}
with
\[ a = \begin{cases} 
      \frac{u_\eps^m - v^m}{u_\eps-v}, & \text{ if } u_\eps \neq v, \\
      0, & \text{ if } u_\eps = v.
   \end{cases}
\]
By defining $a_{k} = \min \{ k, \max\{ a, \frac{1}{k} \} \}$ we have
$0< 1/k \leq  a_{k} \leq k < \infty $, and by $(a_k)_\delta$ we denote
the standard mollification of $a_k$, for a sufficiently small parameter
$\delta=\delta(k)>0$. In this case, to be able to define $(a_k)_\delta$ up to the boundary $\partial_p \Omega_T$, we set $a = 0$ in $\R^{n+1}\setminus \Omega_T$. 
Let $h_k$ be the solution to a backward in time boundary value problem
\begin{equation}\label{eq:dual-problem}
  \left\{
  \begin{array}{ll}
      u_t + (a_k)_\delta \Delta u = 0 & \text{ in } D_{\eps,s}^F, \\
      u = \phi & \text{ on } \mathcal T(D_{\eps,s}^F), \\
      u = 0 & \text{ on } \mathcal S(D_{\eps,s}^F).	
  \end{array}
  \right.
\end{equation}
By the linear theory and the fact that $D_{\eps,s}^F$ can be expressed as in~\eqref{cover-D}, the solution satisfies $h_k \in C^{2,1}(\overline{D_{\eps,s}^F})$.
By Hölder's inequality we obtain 
\begin{align} \label{e.bp-term}
\iint_{D_{\eps,s}^F} &b((h_k)_t + a \Delta h_k) \, \d x \d t = \iint_{D_{\eps,s}^F} b[ a-(a_k)_\delta ]\Delta h_k \, \d x \d t \nonumber \\
&\leq \left( \iint_{D_{\eps,s}^F} b^2 \frac{[(a_k)_\delta - a ]^2}{(a_k)_\delta}\, \d x \d t \right)^\frac{1}{2} \left( \iint_{D_{\eps,s}^F} (a_k)_\delta \left( \Delta h_k \right)^2 \, \d x \d t \right)^\frac{1}{2}.
\end{align}
For the first term on the right-hand side we have 
\begin{align*}
\left\|\frac{b[(a_k)_\delta - a]}{\sqrt{(a_k)_\delta}} \right\|_{L^2(D_{\eps,s}^F)} &\leq \left\|\frac{b[(a_k)_\delta - a_k]}{\sqrt{(a_k)_\delta}} \right\|_{L^2(D_{\eps,s}^F)} + \left\|\frac{b[a_k -a ]}{\sqrt{(a_k)_\delta}} \right\|_{L^2(D_{\eps,s}^F)} \\
&\xrightarrow{\delta \to 0}\left\|\frac{b[a_k -a ]}{\sqrt{a_k}} \right\|_{L^2(D_{\eps,s}^F)}, 
\end{align*}
since $(a_k)_\delta \to a_k$ a.e.,
$\frac{1}{k} \leq (a_k)_\delta \leq k$, $b\in L^2(D_{\eps,s}^F)$ and by using the dominated convergence theorem. We further estimate 
\begin{align*}
\iint_{D_{\eps,s}^F} b^2 \frac{(a_k - a )^2}{a_k}\, \d x \d t &= \iint_{D_{\eps,s}^F \cap \{a < \frac{1}{k}\}} b^2 \frac{(\tfrac{1}{k} - a )^2}{\tfrac{1}{k}}\, \d x \d t \\
&\phantom{+} + \iint_{D_{\eps,s}^F \cap \{a > k\}} b^2 \frac{(a-k )^2}{k}\, \d x \d t \\
&\leq \tfrac{1}{k} \iint_{D_{\eps,s}^F \cap \{a < \frac{1}{k}\}} b^2
\, \d x \d t \\
&\phantom{+} + \tfrac{1}{k} \iint_{D_{\eps,s}^F \cap \{a > k\}} (u_\eps^m - v^m)^2 \, \d x \d t \\
&\le \frac{c_o}{k}
\end{align*}
for a constant $c_o = c_o(m,\|u\|_{L^2(\Omega_T)},\|v\|_{L^2(\Omega_T)})>0$. 
Because of the two preceding estimates, by choosing $\delta>0$ small enough in dependence on $k$,
we achieve
\begin{equation*}
\iint_{D_{\eps,s}^F} b^2 \frac{[(a_k)_\delta - a ]^2}{(a_k)_\delta}\,
\d x \d t
\le
\frac{2c_o}{k}.
\end{equation*}
By using the property that $h_k$ solves the boundary value problem~\eqref{eq:dual-problem}, we have 
\begin{align*}
\iint_{D_{\eps,s}^F} &(a_k)_\delta (\Delta h_k)^2 \, \d x \d t = - \iint_{D_{\eps,s}^F}  (h_k)_t \Delta h_k \, \d x \d t \\
&= \iint_{D_{\eps,s}^F}  h_k (\Delta h_k)_t \, \d x \d t - \int_{\mathcal T(D_{\eps,s}^F)} \phi \Delta \phi \, \d x + \int_{\mathcal B(D_{\eps,s}^F)} h_k \Delta h_k \, \d x \\
&= \iint_{D_{\eps,s}^F}  h_k \Delta (h_k)_t \, \d x \d t + \int_{\mathcal T(D_{\eps,s}^F)} |\nabla \phi|^2 \, \d x - \int_{\mathcal B(D_{\eps,s}^F)} |\nabla h_k|^2 \, \d x \\
&\leq \iint_{D_{\eps,s}^F}  (h_k)_t \Delta  h_k  \, \d x \d t - \iint_{\mathcal S(D_{\eps,s}^F)} (h_k)_t \partial_n h_k \, \d \sigma \d t + \int_{\mathcal T(D_{\eps,s}^F)} |\nabla \phi|^2 \, \d x \\
&= - \iint_{D_{\eps,s}^F} (a_k)_\delta (\Delta h_k)^2 \, \d x \d t + \int_{\mathcal T(D_{\eps,s}^F)} |\nabla \phi|^2 \, \d x
\end{align*}
by integrating by parts in time and space, and using the fact that $h_k \in C^{2,1}(\overline{D_{\eps,s}^F})$ together with the boundary conditions. This implies
$$
\iint_{D_{\eps,s}^F} (a_k)_\delta (\Delta h_k)^2 \, \d x \d t
\leq
\frac12 \int_{\mathcal T(D_{\eps,s}^F)} |\nabla \phi|^2 \, \d x
=
\frac12 \int_{\mathcal{T}(D_{\eps_o,s}^{F,1})\cap(\Omega\times\{s\})} |\nabla \phi|^2 \, \d x,
$$
by the choice  of $\phi$. Collecting the previous results
in~\eqref{e.bp-term}, we deduce  
\begin{equation*}
  \iint_{D_{\eps,s}^F} b((h_k)_t + a \Delta h_k) \, \d x \d t
  \le
  \frac{c_1}{\sqrt{k}},
\end{equation*}
for a constant $c_1 = c_1(m,\|u\|_{L^2(\Omega_T)},\|v\|_{L^2(\Omega_T)}, \|\nabla \phi\|_{L^2\left(\Omega \times \{s\} \right)}) > 0$. 
For the term with $f$ in~\eqref{e.comparison-est}, we estimate
$$
\iint_{D_{\eps,s}^F} \nabla f \cdot \nabla h_k \, \d x \d t \leq \frac{ 1 -(1+\eps)^{m-1} }{(1+\eps)^m} \|\nabla u^m\|_{L^2(\Omega_T)} \left( \iint_{D_{\eps,s}^F} |\nabla h_k|^2 \, \d x \d t \right)^\frac{1}{2}.
$$
By using the fact that $h_k$ is a solution to the linear boundary value problem, with a cutoff-function $\alpha= \alpha (t)$ with properties $\alpha(0) = 1/2$, $\alpha(s) = 1$ and $\alpha' = \frac{1}{2s}$ we obtain
\begin{align*}
0 &= \iint_{D_{\eps,s}^F}  \alpha (h_k)_t \Delta h_k \, \d x \d t + \iint_{D_{\eps,s}^F} \alpha (a_k)_\delta (\Delta h_k)^2  \, \d x \d t \\
&= - \iint_{D_{\eps,s}^F} \alpha \nabla (h_k)_t \cdot \nabla h_k \, \d x \d t + \iint_{D_{\eps,s}^F} \alpha (a_k)_\delta (\Delta h_k)^2 \, \d x \d t \\
&= - \frac{1}{2} \iint_{D_{\eps,s}^F} \alpha  (|\nabla h_k|^2)_t \, \d x \d t + \iint_{D_{\eps,s}^F} \alpha (a_k)_\delta (\Delta h_k)^2 \, \d x \d t \\
&= \frac{1}{2} \iint_{D_{\eps,s}^F} \alpha' |\nabla h_k|^2 \, \d x \d t - \frac12 \int_{\mathcal T(D_{\eps,s}^F)} \alpha |\nabla h_k|^2 \, \d x \\
&\phantom{+} + \frac12 \int_{\mathcal B(D_{\eps,s}^F)} \alpha |\nabla h_k|^2 \, \d x + \iint_{D_{\eps,s}^F} \alpha (a_k)_\delta (\Delta h_k)^2 \, \d x \d t \\
&\geq \frac{1}{4s} \iint_{D_{\eps,s}^F} |\nabla h_k|^2 \, \d x \d t - \frac12 \int_{\mathcal T(D_{\eps,s}^F)} |\nabla\phi|^2 \, \d x,
\end{align*}
which implies 
$$
\left( \iint_{D_{\eps,s}^F} |\nabla h_k|^2 \, \d x \d t \right)^\frac12 \leq  (2s)^\frac12 \left( \int_{\mathcal T(D_{\eps,s}^F)} |\nabla\phi|^2 \, \d x \right)^\frac12.
$$
Therefore, we arrive at the  bound
\begin{align*}
\iint_{D_{\eps,s}^F} \nabla f \cdot \nabla h_k \, \d x \d t \leq (2s)^\frac12 \frac{1-(1+\eps)^{m-1}}{(1+\eps)^m} \|\nabla u^m\|_{L^2(\Omega_T)} \left( \int_{\mathcal T(D_{\eps,s}^F)} |\nabla \phi|^2 \, \d x \right)^\frac12.
\end{align*}
Now, taking also the indices $i$ into account and recalling the fact $b \leq 0$ on $\partial D_{\eps,s}^F \cap Q_i$, we can write~\eqref{e.comparison-est} in the form
\begin{align*}
\int_{\mathcal T(D_{\eps,s}^{F,i})} (u_\eps -v)\phi \, \d x &\leq \frac{c_1}{\sqrt{k}} + (2s)^\frac12 \frac{1-(1+\eps)^{m-1}}{(1+\eps)^m} \left( \int_{\mathcal T(D_{\eps,s}^{F,i})} |\nabla \phi|^2 \, \d x \right)^\frac12 \\
&\phantom{+} + \int_{\mathcal B(D_{\eps,s}^{F,i})\cap (\Omega \times \{\tau_1^i\})} (u_\eps - v)h_k \, \d x. 
\end{align*}
Observe that 
\begin{align*}
\int_{\mathcal B(D_{\eps,s}^{F,i})\cap (\Omega \times \{\tau_1^i\})} (u_\eps - v)h_k \, \d x &\leq \int_{\mathcal B(D_{\eps,s}^{F,i})\cap (\Omega \times \{\tau_1^i\})} (u_\eps-v)_+ h_k \, \d x \\
&\leq \|\phi\|_{L^\infty(\mathcal T(D_{\eps,s}^{F,i}))} \int_{\Omega \times \{\tau_1^i\}} (u_\eps-v)_+ \, \d x
\end{align*}
by using the comparison principle for $h_k$. 
Recall that we chose $\phi$ with the property $\phi=0$ in
$\Omega\times(0,s]\setminus \overline{D_{\eps_o,s}^{F,1}}$.
Therefore, after passing to the limit $k \to \infty$ we infer the bound
\begin{align*}
\int_{\overline{D_{\eps_o,s}^{F,1}}\times(\Omega\times\{s\})} (u_\eps -v)\phi \, \d x &\leq (2s)^\frac12 \frac{1-(1+\eps)^{m-1}}{(1+\eps)^m} \left( \int_{\mathcal T(D_{\eps_o,s}^{F,1})} |\nabla \phi|^2 \, \d x \right)^\frac12 \\
&\phantom{+} + \|\phi\|_{L^\infty(\mathcal T(D_{\eps_o,s}^{F,1}))} \int_{\Omega \times \{\tau_1^i\}} (u-v)_+ \, \d x.
\end{align*}
By passing to the limit $i \to \infty$, the last term
vanishes, since $u \in C([0,T];L^{m+1}(\Omega))$ and
$(u-v)_+(\cdot,\tau_1^i)\le u(\cdot,\tau_1^i)\to 0$ as 
$\tau_1^i\to t_1$. 
In the resulting inequality, we let $\eps\downarrow0$ and conclude
$$
\int_{\overline{D_{\eps_o,s}^{F,1}} \cap (\Omega \times \{s\})} (u-v)\phi \, \d x \leq 0
$$
for every test function $\phi$ as in~\eqref{choice-test-function},   
i.e. that $u \leq v$ a.e. on $\overline{D_{\eps_o,s}^{F,1}} \cap (\Omega
\times \{s\})$.
In particular, on the set $\overline{D_{\eps_o,s}^{1}} \cap (\Omega
\times \{s\})$ we have $0<u\le v\le \frac{u}{1+\eps_o}$, which is not
possible. Therefore, this set must be empty
for every $\eps_o>0$, which implies $u\le v$ a.e. on $\Omega\times\{s\}$.
Since $s\in (t_1,t_2)$ is arbitrary, we obtain $u \leq v$ a.e in $\Omega
\times (t_1,t_2)$. Since $u$ is continuous and $v = \hat R_\psi$ is a supercaloric function, Theorem~\ref{t.supercal-essliminf} implies that $u \leq v$ everywhere in $\Omega \times (t_1,t_2)$.
This can be deduced in any connected component $\Omega \times
\Lambda_i$ of the  positivity set of $u$. Outside these sets we
trivially have $u = 0 \leq v$, which implies that $u \leq v= \hat R_\psi$ everywhere in $\Omega_T$. Since $u$ itself is a continuous weak supersolution, and thus a supercaloric function by Lemma~\ref{l.weasuper-is-supercal} satisfying $u \geq \psi$, it follows that $R_\psi \leq u$ everywhere. Now we have $u \leq \hat R_\psi \leq R_\psi \leq u$ everywhere, which concludes the result for each connected component of $\Omega$. Thus the claim follows.
\end{proof}

From the proof of Theorem~\ref{t.weakvar-is-minimal} we can extract the following comparison result, cf.~\cite[Theorem 3.1]{AvLu}.

\begin{proposition}
Suppose that $0<m<1$. Let $K \Subset \Omega_T$ be a compact set and $\Omega \Subset \R^n$ be a $C^{2,\alpha}$-domain. Suppose that $u \in C(\Omega_T)$ is a supercaloric function in $\Omega_T$ such that $u$ is a weak solution in $\Omega_T \setminus K$ with $u^m \in L^2(0,T;H^1(\Omega)) \cap L^\frac{2}{m}(\Omega_T)$ and taking initial values $u_o$ in $L^1$-sense. Let $v$ be a supercaloric function in $\Omega_T$ satisfying $v^m\in L^2(0,T;H^1(\Omega)) \cap L^\frac{2}{m}(\Omega)$ and taking the initial values $v_o$ in $L^1$-sense. If $u_o \leq v_o$ a.e. on $\Omega$, $u^m(\cdot,t) \leq v^m(\cdot,t)$ on $\partial \Omega$ in the sense of $H^1$-trace for a.e. $t \in (0,T)$ and $u \leq v$ on $K´$, then $u \leq v$ everywhere in $\Omega_T$.
\end{proposition}

As an immediate consequence of Theorem~\ref{t.weakvar-is-minimal} we can conclude the following comparison principle for weak solutions to the obstacle problem.

\begin{corollary}
Let $\Omega$ be a $C^{2,\alpha}$-domain and $\psi_1$ and $\psi_2$ be obstacles with
compact support in $\Omega_T$ satisfying~\eqref{a.obstacle_weak}
and~\eqref{e.holder-obstacle}, and let $u_1$ and $u_2$ be corresponding weak solutions to the obstacle problems with zero boundary values on $\partial_p \Omega_T$ in the sense of Definition~\ref{d.variational-obstacle}. If $\psi_1 \leq \psi_2$, then $u_1 \leq u_2$.
\end{corollary}

\appendix

\section{On the notions of variational solution and weak solution to the obstacle problem} \label{appendix-a}

In Theorem~\ref{t.existence} we use an existence result for the
obstacle problem. In order to cover the full fast diffusion range
(i.e., also the subcritical range) we apply the result given
in~\cite{Schaetzler2}. The aforementioned existence result is
established for the notion of variational solution (see
Definition~\ref{d.obstacle-varsol}). Here we show that a variational solution is also a weak solution to the obstacle problem according to Definition~\ref{d.variational-obstacle} provided that the obstacle is regular enough.

Suppose that the obstacle $\psi$ satisfies
\begin{equation} \label{eq:g-cond}
\begin{aligned}
&\psi^m \in L^2(0,T;H^1(\Omega)),\quad \partial_t \psi^m \in L^\frac{m+1}{m}(\Omega_T), \\ 
&\psi \in C([0,T];L^{m+1}(\Omega)),\quad \text{ and } \quad \psi_o^m := \psi^m(\cdot,0) \in L^{\frac{m+1}{m}}(\Omega) \cap H^1(\Omega).
\end{aligned}
\end{equation}
For a variational solution $u$, suppose
\begin{equation} \label{eq:u-class}
u \in L^\infty(0,T; L^{m+1}(\Omega)), \quad u^m \in L^2(0,T;H^1(\Omega)),\quad  u\geq \psi \, \text{ a.e. in } \Omega_T,
\end{equation}
and that it attains the same boundary and initial values as $\psi$ in
the sense 
\begin{equation} \label{eq:boundary-data}
u^m - \psi^m \in L^2(0,T;H^1_0(\Omega)), \quad 
\lim_{h\to 0} \bint_0^h \int_\Omega |u(x,t) - \psi_o|^{m+1} \, \d x \d t = 0.
\end{equation}

Denote 
\begin{align*}
I(u,v) &:= \tfrac{1}{m+1}\left( u^{m+1} - v^{m+1} \right) - v^m (u-v) \\
&=
\tfrac{1}{m+1} u^{m+1} + \tfrac{m}{m+1} v^{m+1} -  v^mu.
\end{align*}
Now we recall the notion of solution used in \cite{Schaetzler2}.
\begin{definition}\label{d.obstacle-varsol}
Let $\psi$ satisfy~\eqref{eq:g-cond}. A measurable map $u \colon \Omega_T \to \R_{\geq0}$ satisfying~\eqref{eq:u-class} and~\eqref{eq:boundary-data}
is called a variational solution to the obstacle problem for the PME if 
\begin{align}\label{var-ineq-appendix}
\frac12 \iint_{\Omega_\tau} |\nabla u^m|^2 \, \d x \d t 
&\leq 
\iint_{\Omega_\tau} \partial_t v^m(v - u) \, \d x \d t + \frac12 \iint_{\Omega_\tau}|\nabla v^m|^2 \, \d x \d t \\\nonumber
&\phantom{+} - \int_{\Omega} I(u(\tau),v(\tau))\, \d x +\int_{\Omega} I (\psi_o,v(0)) \, \d x \nonumber
\end{align}
for a.e. $\tau \in [0,T]$ and every comparison map $v$ satisfying $v^m \in \psi^m + L^2(0,T; H^1_0(\Omega))$ with $\partial_t v^m \in L^1(0,T;L^\frac{m+1}{m}(\Omega))$, $v(0) \in L^{m+1}(\Omega)$ and $v \geq \psi$ a.e. in $\Omega_T$.
\end{definition}

As an intermediate step, we will prove that every variational solution
is a solution in the following sense. 

\begin{definition} \label{d.obstacle-wsol}
Let $\psi$ satisfy~\eqref{eq:g-cond}. A measurable map $u \colon \Omega_T \to \R_{\geq0}$ satisfying~\eqref{eq:u-class} and~\eqref{eq:boundary-data}$_1$
is called a weak solution to the obstacle problem for the PME if 
\begin{equation}\label{global-var-ineq}
\llangle \partial_t u , \alpha (v^m-u^m) \rrangle_{\psi_o} +
\iint_{\Omega_T} \alpha \nabla u \cdot \nabla \left(v-u \right) \, \d
x \d t \geq 0
\end{equation}
for every cutoff function $\alpha\in C^{0,1}([0,T]; \R_{\geq 0})$ with $\alpha(T) = 0$ and all comparison maps $v$ satisfying $v^m \in \psi^m + L^2(0,T; H^1_0(\Omega))$ with $\partial_t v^m \in L^\frac{m+1}{m}(\Omega_T)$, $v \in C([0,T];L^{m+1}(\Omega))$ and $v \geq \psi$ a.e. in $\Omega_T$. We denoted

\begin{align*}
\llangle \partial_t u, \alpha (v^m-u^m) \rrangle_{\psi_o} := &\iint_{\Omega_T}  \left[ \alpha' \left( \tfrac{1}{m+1}u^{m+1} -  u v^m \right) - \alpha u \partial_t v^m \right] \, \d x \d t \\
&\phantom{=} + \alpha(0) \int_\Omega  \left[ \tfrac{1}{m+1} \psi_o^{m+1} - \psi_o v^m(\cdot,0) \right] \, \d x  .
\end{align*}
\end{definition}

By a simple change of variables ($q = \frac{1}{m}$)~\cite[Theorem
1.2]{Schaetzler2} implies that solutions according to
Definition~\ref{d.obstacle-varsol} exists. Then, we show that the
solution for which the existence is guaranteed, is actually a solution according to Definition~\ref{d.obstacle-wsol} and finally, according to Definition~\ref{d.variational-obstacle}.

\begin{lemma} \label{lem:varsol-is-weaksol}
Let $u$ be a variational solution according to
Definition~\ref{d.obstacle-varsol}. Then $u$ is a weak solution to the
obstacle problem according to Definition~\ref{d.obstacle-wsol}.
\end{lemma}

\begin{proof}

Let $v$ be a comparison map and $\eta,\alpha$ cutoff functions
satisfying the assumptions given in Definition~\ref{d.obstacle-wsol}. Without loss of generality we may assume that $0 \leq \alpha\leq 1$.

Let us use a comparison map 
$$
v_h^m = s v^m + (1-s) (\mollifytime{u^m}{h,\psi_o} - \mollifytime{\psi^m}{h,\psi_o} + \psi^m), 
$$
where we denote $s = s(t) := \eps \alpha(t)$
with $\eps \in (0,1)$, and where we used the time
mollification introduced in~\eqref{eq:time-mollif}. Observe that $v_h$ is an admissible comparison map. In particular $v_h \geq \psi$ as convex combination of functions enjoying that property and $v_h^m -\psi^m \in L^2(0,T;H^1_0(\Omega))$.

First, suppose that $\alpha$ vanishes in a neighbourhood of $T$ and let $\tau \in (0,T)$ be so large that $\alpha(t) = 0$ for all $t \in [\tau,T]$.
For the parabolic term in~\eqref{var-ineq-appendix} we have 
\begin{align*}
\iint_{\Omega_\tau} &\partial_t v_h^m (v_h - u) \, \d x \d t \\
&=
\tfrac{m}{m+1} \int_{\Omega} v_h^{m+1}\, \d x \bigg|_0^\tau \\
&\phantom{+} -
\iint_{\Omega_\tau}u \left(  \partial_t(s v^m) + \partial_t [(1-s) \mollifytime{u^m}{h,\psi_o}] + \partial_t [ (1-s)(\psi^m - \mollifytime{\psi^m}{h,\psi_o} )] \right) \, \d x \d t. 
\end{align*}
A part of the second term on the right hand side can be estimated by 
\begin{align*}
- \iint_{\Omega_\tau} u (1-s) \partial_t \mollifytime{u^m}{h,\psi_o} \, \d x \d t 
&\leq
- \tfrac{m}{m+1} \iint_{\Omega_\tau}(1-s) \partial_t \mollifytime{u^m}{h,\psi_o}^\frac{m+1}{m} \, \d x \d t \\
&= - \tfrac{m}{m+1} \int_{\Omega}(1-s)  \mollifytime{u^m}{h,\psi_o}^\frac{m+1}{m} \, \d x \bigg|_0^\tau \\
&\phantom{+} -
\tfrac{m}{m+1} \eps \iint_{\Omega_\tau} \alpha' \mollifytime{u^m}{h,\psi_o}^\frac{m+1}{m}\, \d x \d t.
\end{align*}
Since according to Lemma~\ref{lem:mollifier}\,(i), we have 
$$
\int_0^T \|\mollifytime{u^m}{h,\psi_o}(t) - u^m(t)\|^\frac{m+1}{m}_{L^\frac{m+1}{m}(\Omega)}\, \d t \xrightarrow{h\to 0} 0,
$$
it follows that along a subsequence $\|\mollifytime{u^m}{h,\psi_o}(t)
- u^m(t)\|_{L^\frac{m+1}{m}(\Omega)} \xrightarrow{h\to
  0} 0$ for a.e. $t \in (0,T)$. Let $\tau$ be an instant of time at which this convergence holds. Then, by passing to the limit $h\to 0$ (along the aforementioned subsequence) we obtain
\begin{align*}
\limsup_{h\to0} &\left( \iint_{\Omega_\tau} \partial_t v_h^m(v_h - u) \, \d x \d t - \int_{\Omega} I(u(\tau),v_h(\tau))\, \d x +\int_{\Omega} I (\psi_o,v_h(0)) \, \d x \right) \\
&\leq 
\eps \iint_{\Omega_\tau} \left[ \alpha' \left( \tfrac{1}{m+1}u^{m+1} -  u v^m \right) - \alpha u \partial_t v^m \right] \, \d x \d t \\
&\phantom{=} + \eps \alpha(0) \int_\Omega  \left[ \tfrac{1}{m+1} \psi_o^{m+1} - \psi_o v^m(\cdot,0) \right] \, \d x  .
\end{align*}

We also have that 
\begin{align*}
\iint_{\Omega_\tau} |\nabla v_h^m|^2 \, \d  x \d t \xrightarrow{h\to 0} \iint_{\Omega_\tau} |\nabla u^m + \eps\alpha \nabla(v^m-u^m) |^2 \, \d x\d t.
\end{align*}
By dividing every term in~\eqref{var-ineq-appendix} by $\eps> 0$, for the gradient terms we have 
\begin{align*}
\iint_{\Omega_\tau} &\frac{1}{2\eps}(|\nabla u^m + \eps\alpha \nabla(v^m-u^m)|^2 -|\nabla u^m|^2) \, \d x \d t \\
&\xrightarrow{\eps \to 0} 
\iint_{\Omega_\tau} \alpha \nabla u^m \cdot \nabla(v^m-u^m)\, \d x \d t.
\end{align*}
By combining all the estimates and passing to the limit $\tau \to T$,
the claim follows with $\alpha$ vanishing in a neighbourhood of
$T$. For $\alpha$ that is required to satisfy only $\alpha(T) = 0$ it
is deduced by a standard cutoff argument.
\end{proof}

\begin{lemma} \label{lem:global-wsol-is-local}
Suppose that $u$ is a weak solution to the obstacle problem according
to Definition~\ref{d.obstacle-wsol}, then $u$ is a weak solution to
the obstacle problem with $u=\psi$ on $\partial_p\Omega_T$ according to Definition~\ref{d.variational-obstacle}.
\end{lemma}

\begin{proof}
Observe that~\cite[Lemma 5.2]{BLS} implies $u \in
C([0,T];L^{m+1}(\Omega))$ and $u(\cdot,0) = \psi_o$ a.e. in
$\Omega$. Therefore, it only remains to prove the variational
inequality~\eqref{e.local_var_eq}. 

Let $v\in K'_\psi(\Omega_T)$ and $\eta \in
C_0^1(\Omega,\R_{\geq0})$. Without loss of generality, we may assume
$0 \leq \eta \leq 1$. In the variational
inequality~\eqref{global-var-ineq},
we use a test function 
$$
v_h^m = \eta v^m + (1-\eta) (\mollifytime{u^m}{h,\psi_o} - \mollifytime{\psi^m}{h,\psi_o} + \psi^m),
$$
which satisfies the assumptions from Definition~\ref{d.obstacle-wsol}.
For the divergence part we have
\begin{align}\label{eq:conv-div-part}
\iint_{\Omega_T}& \alpha \nabla u^m \cdot \nabla (v_h^m - u^m) \, \d x
                 \d t\\\nonumber
  &= \iint_{\Omega_T} \alpha \nabla u^m \cdot \nabla (\eta (v^m-\mollifytime{u^m}{h,\psi_o})) \, \d x \d t \\\nonumber
&\phantom{+}+ \iint_{\Omega_T} \alpha \nabla u^m\cdot \nabla ( \mollifytime{u^m}{h,\psi_o} - u^m) \, \d x \d t \\\nonumber
&\phantom{+} + \iint_{\Omega_T} \alpha \nabla u^m\cdot \nabla ((1-\eta) (\psi^m - \mollifytime{\psi^m}{h,\psi_o})) \, \d x \d t \\\nonumber
&\xrightarrow{h\to 0} \iint_{\Omega_T} \alpha \nabla u^m\cdot \nabla (\eta (v^m - u^m)) \, \d x \d t.
\end{align}
Here we used Lemma~\ref{lem:mollifier}\,(i) to pass to the limit. Moreover, we may estimate 
\begin{align*}
\iint_{\Omega_T} - \alpha u \partial_t v_h^m \, \d x\d t &= \iint_{\Omega_T}-\alpha \eta u \partial_t v^m \, \d x \d t \\ 
&\phantom{+} +\iint_{\Omega_T} - \alpha (1-\eta) u \partial_t \mollifytime{u^m}{h,\psi_o} \, \d x \d t \\
&\phantom{+} + \iint_{\Omega_T} -\alpha (1-\eta) u \partial_t (\psi^m - \mollifytime{\psi^m}{h,\psi_o} ) \, \d x \d t \\
&\leq \iint_{\Omega_T}-\alpha \eta u \partial_t v^m \, \d x \d t \\ 
&\phantom{+} + \tfrac{m}{m+1}\iint_{\Omega_T} \alpha' (1-\eta) \mollifytime{u^m}{h,\psi_o}^\frac{m+1}{m} \, \d x \d t \\
&\phantom{+} + \tfrac{m}{m+1} \alpha(0) \int_{\Omega} (1-\eta) \psi_o^{m+1} \, \d x \\ 
&\phantom{+} + \iint_{\Omega_T} -\alpha (1-\eta) u \partial_t (\psi^m - \mollifytime{\psi^m}{h,\psi_o} ) \, \d x \d t .
\end{align*}
Observe that also
$$
-\psi_o v_h^m(\cdot,0) = -\eta \psi_o v^m(\cdot,0) - (1-\eta) \psi_o^{m+1}.
$$
Now by combining the estimates and passing to the limit $h \to0$, we obtain
\begin{align*}
\limsup_{h\to 0} \llangle \partial_t u, \alpha (v_h^m-u^m) \rrangle_{\psi_o} &\leq \iint_{\Omega_T} \eta \left[ \alpha' \left( \tfrac{1}{m+1}u^{m+1} -  u v^m \right) - \alpha u \partial_t v^m \right] \, \d x \d t \\
&\phantom{=} + \alpha(0) \int_\Omega \eta \left[ \tfrac{1}{m+1} \psi_o^{m+1} - \psi_o v^m(\cdot,0) \right] \, \d x.
\end{align*}
For cutoff functions $\alpha$ with compact support in $(0,T)$, as they
are considered in~\eqref{e.local_var_eq}, the last integral
vanishes. Consequently, for such cutoff functions the last formula becomes
\begin{equation*}
  \limsup_{h\to 0} \llangle \partial_t u, \alpha (v_h^m-u^m)
  \rrangle_{\psi_o}
  \le
  \llangle \partial_t u, \alpha\eta (v_h^m-u^m) \rrangle.
\end{equation*}
In combination with~\eqref{eq:conv-div-part}, this yields the
variational inequality~\eqref{e.local_var_eq} and concludes the proof.
\end{proof}

\end{document}